\documentclass[a4paper,10pt]{amsart}
\usepackage[utf8]{inputenc}
\usepackage{amsthm}
\usepackage{amsmath}
\usepackage{bm}
\usepackage{amsfonts}
\usepackage{amssymb}
\usepackage{mathtools}
\usepackage{appendix}
\usepackage{mathrsfs}
\usepackage{setspace}
\usepackage{xcolor}
\usepackage{todonotes}
\usepackage{thmtools} 
\usepackage[foot]{amsaddr}
\usepackage[pdfdisplaydoctitle,colorlinks,breaklinks,urlcolor=blue,linkcolor=blue,citecolor=blue]{hyperref} 

\newcommand{\C}{\mathbb{C}}

\newcommand{\F}{\mathcal{F}}
\newcommand{\G}{\mathcal{G}}
\renewcommand{\H}{\mathbb{H}}

\newcommand{\N}{\mathbb{N}}

\newcommand{\PP}{\mathbb{P}}
\newcommand{\PR}{\mathcal{P}}

\newcommand{\R}{\mathbb{R}}

\newcommand{\T}{\mathbb{T}}

\DeclareMathOperator{\var}{Var}

\DeclareMathOperator{\re}{Re}
\DeclareMathOperator{\im}{Im}
\DeclareMathOperator{\e}{e}

\DeclareMathOperator{\cov}{Cov}

\renewcommand{\epsilon}{\varepsilon}
\renewcommand{\setminus}{\smallsetminus}

\newcommand{\one}{\bm{1}}
\newcommand{\obar}[1]{\overline{#1}}

\newcommand{\de}{\partial}

\newcommand{\set}[1]{\left\{#1\right\}}
\newcommand{\pa}[1]{\left(#1\right)}
\newcommand{\bra}[1]{\left[#1\right]}
\newcommand{\abs}[1]{\left|#1\right|}
\newcommand{\norm}[1]{\left\|#1\right\|}
\newcommand{\brak}[1]{\left\langle#1\right\rangle}

\newcommand{\expt}[2][]{\mathbb{E}_{#1}\left[#2\right]}

\newtheorem{theorem}{Theorem}[section]

\newtheorem{corollary}[theorem]{Corollary}
\newtheorem{lemma}[theorem]{Lemma}
\newtheorem{proposition}[theorem]{Proposition}

\theoremstyle{remark}
\newtheorem{remark}[theorem]{Remark}

\numberwithin{equation}{section}

\title[Nonlinear Functionals of Hyperbolic Random Waves]{Nonlinear Functionals of Hyperbolic \\ Random Waves: the Wiener Chaos Approach}
\author[F. Grotto]{Francesco Grotto}
\address{Universit\`a di Pisa, Dipartimento di Matematica, 5 Largo Bruno Pontecorvo, 56127 Pisa, Italia}
\email{\href{mailto:francesco.grotto at unipi.it}{francesco.grotto at unipi.it}}
\author[G. Peccati]{Giovanni Peccati}
\address{Universit\'e du Luxembourg, Maison du Nombre, 6 Avenue de la Fonte, 4364 Esch-sur-Alzette, Luxembourg}
\email{\href{mailto:giovanni.peccati at uni.lu}{giovanni.peccati at uni.lu}}
\keywords{Random Waves, Hyperbolic Space, Wiener Chaos}
\date\today

\begin{document}
	
	\begin{abstract}
		We consider Gaussian random waves on hyperbolic spaces and establish variance asymptotics and central limit theorems for a large class of their integral functionals, both in the high-frequency and large domain limits. Our strategy of proof relies on a fine analysis of Wiener chaos expansions, which in turn requires us to analytically assess the fluctuations of integrals involving mixed moments of covariance kernels. Our results complement several recent findings on non-linear transforms of planar and arithmetic random waves, as well as of random spherical harmonics. In the particular case of 2-dimensional hyperbolic spaces, our analysis reveals an intriguing discrepancy between the high-frequency and large domain fluctuations of the so-called {\em fourth polyspectra} --- a phenomenon that has no counterpart in the Euclidean setting. We develop applications of a geometric flavor, most notably to excursion volumes and occupation densities.
	\end{abstract}
	
	\maketitle
	
	\tableofcontents
	
	\section{Introduction}\label{sec:introduction}
	
	\subsection{Overview} The aim of this paper is to initiate the study of non-linear functionals of Gaussian random waves (that is, generalized Gaussian eigenfunctions of the Laplacian) defined on hyperbolic spaces of arbitrary dimension --- with specific emphasis on variance asymptotics and central limit theorems. As put forward in the title, our approach is based on a careful analysis of {\em Wiener chaos expansions}, which we implement by using several non-trivial refinements of the general theory developed in \cite{NPTRF, Nourdin2012}, see Section \ref{sec:intfunc}. One of the main contributions of our work is the derivation of new analytic estimates for covariance kernels of hyperbolic waves (stated in Section \ref{sec:moments}), which will allow us to deal simultaneously both with the high-frequency and large domain asymptotic regimes. We will see that our findings naturally complement several recent studies of Gaussian random waves on manifolds, such as {\em Euclidean random waves} \cite{Berry1977, Berry2002, Dalmao2021, Dalmao2022, Nourdin2019, Peccati2020, Notarnicola2022}, {\em arithmetic random waves} \cite{Cammarota2019, Dalmao2019, KKW2013, Oravecz2008, Rudnick2008} and {\em random spherical harmonics} \cite{Marinucci2011,Marinucci2011a,Marinucci2011book,Marinucci2014,Marinucci2015,Marinucci2021,Rossi2019}. 
	
	\subsection{First definitions} Denote by $\H^n$, $n\geq 2$,  the $n$-dimensional hyperbolic space (that is, the simply connected manifold with constant negative sectional curvature) and let $\lambda\geq (n-1)^2/4$. The {\em hyperbolic random wave} with frequency $\lambda$, written $u_\lambda:= \{u_\lambda(x) : x\in \H^n \}$, is defined as the unique (in distribution) centered and unit variance real Gaussian field on $\H^n$ such that (i) the law of $u_\lambda$ is invariant with respect to the isometries of $\H^n$, (ii) paths of $u_\lambda$ solve a.s. the Laplace-Beltrami eigenvalue problem 
	$\Delta_{\H^n} u_\lambda+\lambda u_\lambda=0$, where $\Delta_{\H^n}$ is the hyperbolic Laplacian (see \autoref{prop:randomwave} for details). 
	
	\smallskip
	
	The random wave $u_\lambda$ is the exact hyperbolic counterpart of the {\em Euclidean random wave} $v_\lambda:= \{v_\lambda(x) : x\in \R^n \}$ (see \cite{Berry1977, Berry2002}), that one can similarly characterize as being the unique centered and unit variance real Gaussian field on $\R^n$ verifying properties (i) and (ii) above, with $\H^n$ and $\Delta_{\H^n}$ replaced, respectively, by $\R^n$ and $\Delta = -\sum_i \partial^2/\partial x_i^2$. Further remarkable examples of non-Euclidean random waves, to which $u_\lambda$ should be compared, are the already discussed random spherical harmonics and arithmetic random waves (that are, respectively, Laplace eigenfunctions on the sphere $S^n$ and on the flat torus $\T^n$), as well as the class of Gaussian {\em monochromatic random waves} on general compact manifolds \cite{Canzani2020, DNPR, Zelditch2009}.
	We also recall that hyperbolic random waves appeared in another guise in \cite{Cohen2012}, in the context of spectral decomposition of stationary
	Gaussian fields on $\H^n$ (the latter as a particular case of homogeneous space).
	
	\begin{remark} \begin{itemize}
			\item[(i)] For future reference, we recall that (as an application e.g. of \cite[Theorem 5.7.2]{Adler2007}) the above characterization of $v_\lambda$ is equivalent to requiring that, for $x,y\in \R^n$,
			\begin{multline} \label{eq:berrycov}
			\expt{v_\lambda(x)v_\lambda(y)}=C_{n,\lambda}(x,y)
			\coloneqq\frac1{\omega_{n-1}}\int_{S^{n-1}} e^{i \sqrt\lambda u \cdot (x-y)}du\\
			=\frac{(2\pi)^{n/2}}{\omega_{n-1}} \pa{\sqrt\lambda |x-y|}^{1-n/2} J_{n/2-1}(\sqrt \lambda |x-y|),
			\end{multline}
			where $\omega_{n-1}$ is the hypersurface volume of $S^{n-1}$ and $J_\nu$ is the Bessel function of order $\nu$ (see e.g. \cite[Chapter 10]{Olver2010}); we also point out that an analogous representation of the covariance of $u_\lambda$ will emerge from the statement of \autoref{prop:randomwave} below. 
			\item[(ii)] The central role played by Euclidean random waves in the probabilistic analysis of Laplace eigenfunctions is amplified by the so-called {\em Berry's random wave conjecture} --- originally formulated in \cite{Berry1977} --- according to which the unit energy random wave $v_1$ is a universal model for the high-frequency local behavior of deterministic Laplace eigenfunctions on chaotic billiards, among which negatively curved manifolds are paradigmatic examples. We refer the reader to \cite{Jain2017} for a discussion of the role of random wave models in the physical literature, and to \cite{Abert2018, Ingremeau2020} for mathematically rigorous approaches toward Berry's conjecture. See also \cite{Canzani2020, DNPR}, as well as Section \ref{ssec:curvature} below.
		\end{itemize}
	\end{remark}

	\smallskip
	
	\subsection{Motivation and background} As discussed e.g. in the survey \cite{Wigman2022} (to which we refer the reader for an exhaustive list of references), in recent years considerable attention has been devoted to {\em local geometric functionals} associated with level sets of random waves, such as {\em excursion volumes}, {\em occupation measures} and {\em volumes of level sets} --- among which {\em nodal volumes} (i.e., the Haussdorff measures of zero loci) play a pivotal role. 
	
	\smallskip
	
	A remarkable phenomenon is that in a number of crucial cases (see \cite{Cammarota2016, Marinucci2021, Marinucci2015, Marinucci2016, Nourdin2019, Peccati2018, Peccati2020, Notarnicola2022} for a sample) the study of these geometric functionals can be fruitfully reduced to the asymptotic analysis of their orthogonal projections on {\em Wiener chaoses}, as formally defined in Section \ref{sss:chaos}. Such a strategy --- which corresponds to the ``Wiener chaos approach'' advertised in the title --- is described in detail in the forthcoming Section \ref{ss:wienerchaosstrategy} and relies pervasively on the abstract theory of probabilistic approximations presented in \cite{Nourdin2012}; see also \cite{KL2001, Slud1991, Slud1994} for some earlier use of Wiener chaos in the geometric study of Euclidean Gaussian fields\footnote{ Here, an important caveat is that the covariance structure of random waves typically {\em does not} satisfy the integrability assumptions required in order to directly apply the results from \cite{KL2001, Slud1991, Slud1994}, in such a way that, for random waves, several ad hoc arguments have to be developed on a case-by-case basis.}. 
	
	The main contribution of our work consists in the first explicit application of Wiener chaos techniques to a class of integral functionals associated with non-Euclidean random waves on non-compact manifolds, thus setting the bases for the asymptotic analysis of more general geometric quantities. 

	\subsection{Main contributions} The principal focus of our paper is on integral functionals of the form
	\begin{equation*}
	\G(u_\lambda)=\int_{B_R} G(u_\lambda(x)) dm_n(x), 
	\quad B_R\subset \H^n, \quad G:\R\to\R,
	\end{equation*} 
	where $m_n$ is the hyperbolic volume, and $B_R$ is a ball of radius $R$ in the hyperbolic distance.
	Most of our efforts will be devoted to the study of those functionals $\G(u_\lambda)$ (typically called {\em polyspectra}) obtained by taking $G$ to be a Hermite polynomial of a fixed order, whose behavior is investigated in two different limiting regimes: 
	high-energy ($\lambda\to\infty$) and large domain ($R\to\infty$).
	Our main results, stated in full detail in Section \ref{ssec:clts},
	yield variance estimates and Central Limit Theorems (CLTs) for polyspectra of arbitrary orders, from which one can deduce CLTs for functionals $\G(u_\lambda)$ associated with a generic $G$. 
	
	\smallskip

	\begin{remark} Theorem \ref{thm:cltHermite} --- which is one of the main contributions of the present work --- will reveal an interesting phenomenon, namely: whereas in the high-frequency regime the asymptotic behavior of hyperbolic and Euclidean polyspectra roughly coincide, the same conclusion {\em does not hold} in the large-domain limit in the case $n=2$. In the parlance of time-series analysis, such a result seems to indicate that, unlike Euclidean random waves, hyperbolic random waves on $\H^2$ display a form of {\em short memory}, see e.g. \cite{Doukhan2003, Nourdinfbm} for an introduction to this concept.
	\end{remark}
	
	In Sections \ref{ssec:excursion} and \ref{ssec:leray}, we will apply our results to study two remarkable functionals associated with the excursions of $u_\lambda$:
	
	\begin{itemize}
		\item[(a)] the \emph{volume of the excursion set}
		\begin{equation}\label{eq:excursionset}
		m_n \pa{\set{x\in B_R:u_\lambda(x)> t}}=\int_{B_R} \one_{u_\lambda> t}(x)dm_n(x);
		\end{equation}
		
		\item[(b)] the \emph{Leray measure}
		\begin{equation}\label{eq:leraymeasure}
		L_{R,\lambda}\coloneqq \lim_{\epsilon\to 0} \frac1{2\epsilon} m_n\pa{\set{x\in B_R: |u_\lambda(x)|\leq \epsilon}},
		\end{equation}
		which can be formally understood as the integral of a generalized function, as follows:
		\begin{equation*}
		L_{R,\lambda}=\int_{B_R} \delta_0(u_\lambda(x)) dm_n(x).
		\end{equation*}
	\end{itemize}
	\begin{remark}
		
		With probability one, the nodal set of $u_\lambda$ is a submanifold of codimension $1$.
		As a consequence a -- perhaps more natural -- local functional to consider is the induced $(n-1)$-dimensional volume.
		However, a functional such as the \emph{nodal length} in dimension $n=2$,
		(formally) given by
		\begin{equation*}
		\text{length}(\set{x\in B_R: u_\lambda(x)=0})
		=\int_{B_R} \delta_0(u_\lambda(x)) \sqrt{\brak{d u_\lambda, d u_\lambda}_{T^\ast_x\H^n}} dm_n(x),
		\end{equation*}
		also involves the differential $du_\lambda$ of the random field, making its study not directly achievable by the techniques of the present paper. We prefer to regard such an issue as a separate topic and defer it to future investigations. 
	\end{remark}
	
	\subsection{Structure} In Section \ref{sec:Hd}, we recall the necessary preliminaries on 
	geometry and spectral theory of $\H^n$, and then rigorously introduce the hyperbolic random wave model. Section
	\ref{sec:intfunc} contains a discussion of our main results on functionals of random waves. Finally,
	Section \ref{sec:moments} is devoted to the technical core of the proofs.
	
	\subsection{Notation}
	We write $X\sim Y$ when random variables $X,Y$ --taking values in the same space-- have the same law. We write $N(\alpha, \beta^2)$ to indicate a Gaussian random variable with mean $\alpha\in \R$ and variance $\beta^2\geq 0$. The term \emph{distribution} will always refer to an element of the dual space of smooth functions
	on some manifold, that is a generalized function, never to the law of a random variable. Landau $O$'s and $o$'s have their usual meaning, subscripts indicating eventual dependence on parameters.
	The symbol $C$ will denote a positive constant, possibly differing in any of its occurrence
	even in the same formula, depending only on eventual subscripts.
	The expression $A\simeq_{a,b} B$ indicates that $B$ is both an upper and lower bound by $A$
	up to strictly positive multiplicative constants depending only on eventual subscripts $a,b$. 
	Expressions $A\lesssim_{a,b} B$ or $A\gtrsim_{a,b} B$ indicate respectively 
	an upper and lower bound in the same sense.
	
	\subsection{Acknowledgments} Research supported by the Luxembourg National Research Fund (Grant: {\bf O21/16236290/HDSA}).
	F. G. acknowledges support of INdAM through the INdAM-GNAMPA Project CUP\_E55F22000270001. 
	
	\section{Geometry of Hyperbolic Space and Random Waves}\label{sec:Hd}
	
	The hyperbolic space $\H^n$ is the simply connected $n$-dimensional Riemannian manifold of constant negative curvature $-1$.
	It is modeled by one sheet of the two-sheeted hyperboloid $x_0^2-x_1^2-\dots -x_n^2=1$ in $\R^{n+1}$, say $x_0>0$,
	with the Riemannian metric being induced by Minkowski metric $-dx_0^2+dx_1^2+\dots+dx_n^2$ on the ambient space $\R^{n+1}$.
	The Riemannian distance in this parametrization $\H^n\ni x=(x_0,x_1,\dots, x_n)$ is given by
	\begin{equation*}
	d(x,y)=\cosh^{-1}\bra{x,y},\quad \bra{x,y}=x_0y_0-x_1y_1-\dots-x_ny_n,\quad x,y\in\H^n.
	\end{equation*}
	We will denote by $dm_n$ the Riemannian volume on $\H^n$, or rather an arbitrarily fixed positive multiple of it,
	such choice being completely irrelevant for our goals; accordingly, we will write for simplicity $L^2(\H^n)=L^2(\H^n,dm_n)$.
	
	Besides Cartesian coordinates of the hyperboloid model, we will often employ polar (geodesic) coordinates 
	$\H^n\ni x=(r,\vartheta)$, $r=d(x,x_0)>0$, $\vartheta\in S^{n-1}$,
	around a given point $x_0\in \H^n$, in terms of which the volume element is given by
	\begin{equation*}
	dm_n(x)=c_n \sinh(r)^{n-1} dr d\varsigma_{n-1}(\vartheta)
	\end{equation*}
	with $d\varsigma_{n-1}(\vartheta)$ denoting\footnote{We prefer the graphic variant $\varsigma$ (`final sigma') since the symbol $\sigma$
		is customarily used to parametrize the Laplacian's spectrum, see the subsequent Section.}
	the volume form on the sphere $S^{n-1}$, $c_n>0$ depending on the arbitrary choice of a positive multiple of the volume $m_n$. We will also employ the usual notation $\omega_{n}=\int_{S^n}d\varsigma_{n}$,
	with $\omega_0=2$.
	
	\medskip
	
	The content of the forthcoming Sections \ref{ssec:spectral} and \ref{ssec:hyperbolicwaves} is classical: the reader is referred to \cite[Section 4]{Strichartz1989} and \cite[Section 2]{Hislop1994} for definitions, proofs, and examples.
	
	\subsection{Spectral Theory of Hyperbolic Space}\label{ssec:spectral}
	
	We will denote by $\Delta=\Delta_{\H^n}$ the Laplace-Beltrami operator on $\H^n$. Since the metric of $\H^n$ is induced by the embedding into Minkowski space, we have a convenient representation of the Laplacian on $\H^n$ in terms of the d'Alembert operator $\square=-\de^2/\de{x_0^2}+\de^2/\de{x_1^2}+\dots+\de^2/\de{x_n^2}$ on the ambient space $\R^{n+1}\supset \H^n$, that is
	\begin{equation}\label{eq:dalembert}
	\Delta_{\H^n}f=\square \left.f\pa{x/\sqrt{[x,x]}}\right|_{x\in\H^n}.
	\end{equation}
	We recall the spectral theorem on the hyperbolic space:

	\begin{theorem}\label{thm:spectrum}\cite[Example 2.11, Theorem 2.12]{Hislop1994}
		The Laplace-Beltrami operator $\Delta$ on $\H^n$, regarded as an unbounded operator on $L^2(\H^n)$ densely defined
		on smooth functions, is essentially self-adjoint; 
		its spectrum is purely absolutely continuous and given by
		\begin{equation}\label{e:parispectrum}
		\left[\pa{\frac{n-1}{2}}^2,\infty\right)=\set{\lambda=\sigma^2+\alpha^2, \, \sigma = \sigma_n =\frac{n-1}{2},\, \alpha\in \R}.
		\end{equation}
		The projection operator on the eigenspace relative to $\lambda=\sigma^2+\alpha^2$ is given by
		\begin{gather*}
		\PR_\lambda f(z)=\omega_{n-1} \rho_n(\alpha) \int_{\H^n} F_{n,\lambda}(d(x,y)) f(y)dm_n(y),\quad f\in L^2(\H^n),\\
		\rho_n(\alpha)=\frac1{(2\pi)^n} \abs{\frac{\Gamma(\sigma+i\alpha)}{\Gamma(\alpha)}}^2,
		\end{gather*}
		which is expressed in terms of the so-called \emph{spherical function} \cite[(4.3)]{Strichartz1989}
		\begin{align}
		\label{eq:Fdef1}
		F_{n,\lambda}(d(x,y))&=\frac1{\omega_{n-1}}\int_{S^{n-1}} 
		\bra{x,(1,\vartheta)}^{-\sigma+i\alpha} \bra{y,(1,\vartheta)}^{-\sigma-i\alpha}d\varsigma_{n-1}(\vartheta),\\
		\label{eq:Fdef2}
		F_{n,\lambda}(r)&=\frac{\omega_{n-2}}{\omega_{n-1}} \int_0^\pi 
		\pa{\cosh r-\sinh r \cos\theta}^{-\sigma+i\alpha} (\sin \theta)^{n-2} d\theta. 
		\end{align}
	\end{theorem}
	
	The projection operators $\PR_\lambda$ naturally satisfy $f(x)=\int_0^\infty \PR_{\sigma^2+\alpha^2} f(x)d\alpha$
	\cite[(4.2)]{Strichartz1989}, and the function $\rho_n(\alpha)$ is thus the \emph{spectral measure} (see \autoref{prop:fourier}).
	Spherical functions take such a name because $\psi(x,y)=F_{n,\lambda}(d(x,y))$ is the unique (real) radial solution of
	the eigenvalue problem
	\begin{equation*}
	\Delta_x \psi(x,y)=\lambda \psi(x,y), \quad \psi(x,x)=1, \quad x,y\in\H^n,
	\end{equation*}
	where $\Delta_x$ indicates an application of the Laplacian $\Delta_{\H^n}$ to the mapping $x\mapsto \psi(x,y)$.
	
	\medskip
	
	\begin{remark}
		In what follows, we will use $F_{n,\lambda}(d(x,y))$ as the covariance function of a Gaussian field.
		In fact, formulas \eqref{eq:Fdef1} and \eqref{eq:Fdef2} define positive-definite functions also
		when $\alpha=0$ and $\sigma\in (0,(n-1)/2]$ \cite[Sec. 5.3]{Cohen2012}, thus one can consider Gaussian fields with such covariance
		for this additional choice of parameters, see \cite{Cohen2012}. We leave the study of these fields open for future research.
	\end{remark}
	
	\begin{remark}[Notational]
		Throughout the paper, the parameters $\lambda,n,\sigma,\alpha$ will always be related through the relations put forward in formula \eqref{e:parispectrum}. In particular, dependence $n$ can be given in terms of $\sigma$ only and,
		and given $n$, a dependence on $\lambda$ can be given in terms of $\alpha^2$ only.
		Equation \eqref{eq:Fdef1} is the prototypical example of this situation: it is easy to observe that the right-hand side
		does not depend on the sign of $\alpha$, so overall dependence is on $\lambda$.
	\end{remark}
	
	Writing the eigenvalue problem in polar coordinates \cite[(4.6)]{Cohen2012} one readily obtains the following ODE satisfied by $F_{n,\lambda}$:
	\begin{gather}\label{eq:Fode}
	\frac{d^2 }{d r^2}F_{n,\lambda}(r)+\frac{n-1}{\tanh r}\frac{d}{d r}F_{n,\lambda}(r)+\lambda F_{n,\lambda}(r)=0,
	\quad r>0,\\ \nonumber
	F_{n,\lambda}(0)=1, \qquad F_{n,\lambda}'(0)=0.
	\end{gather}
	As we recall in Section \ref{sec:moments}, solutions of such ODE can be represented with hypergeometric
	functions; together with the integral representation \eqref{eq:Fdef2} this will allow
	us to obtain good approximations on $F_{n,\lambda}$ on which our arguments heavily rely.
	
	\subsection{Waves on Hyperbolic Spaces}\label{ssec:hyperbolicwaves}
	As recalled in the Introduction, the class of {\em Euclidean plane waves} is the collection of all exponential functions $x\mapsto e^{ix\cdot k} $, $k\in \R^d$, and that each of them trivially verifies the Laplace equation
	\begin{align}\label{e:euclaplace}
	\Delta_{\R^d} e^{ix\cdot k}=|k|^2 e^{ix\cdot k}, \quad x\in\R^d.
	\end{align}
	Euclidean plane waves are \emph{generalized eigenfunctions} of the Laplacian $\Delta_{\R^d}=-\sum^d \partial_j^2$,
	in the sense that they are smooth functions satisfying \eqref{e:euclaplace} but they do not belong to $L^2(\R^d)$
	(in which we set spectral theory). 
	Plane waves as above are indexed by $k\in \R^d$, or equivalently by their wavenumber $|k|\in\R^+$
	(indicating the relative eigenvalue $|k|^2$) and the direction of the wave $k/|k|\in S^{n-1}$.   
	
	On the hyperbolic space $\H^n$, one actually has a perfect analog of plane random waves, that is obtained (for each $n\geq 2$) by considering the smooth functions $x\mapsto \e_n(x,\alpha, u)$ derived from the following mappings on the product space $\H^n\times\R^\ast\times S^{n-1}$:
	\begin{gather*}
	\e_n:\H^n\times\R^\ast\times S^{n-1}\to \C, \quad \e_n(x,\alpha,u)=\bra{x,(1,u)}^{-\sigma +i\alpha},
	\end{gather*}
	in such a way that the following equation is satisfied:
	\begin{gather}\label{e:hyplaplace}
	\Delta_{\H^n} \e_n(x,\alpha,u)= (\sigma^2+\alpha^2)\e_n(x,\alpha,u), \quad x\in \H^n.
	\end{gather}
	Note that in \eqref{e:hyplaplace} the operator $\Delta_{\H^n}$ is applied to the variable $x$; the formula can be directly checked by applying the expression \eqref{eq:dalembert} to $\e_n(x,\alpha,u)$ and carrying through the tedious but elementary computation.
	The functions $\e_n(\cdot,\alpha, u)$ are thus generalized eigenfunctions of the Laplace-Beltrami operator $\Delta=\Delta_{\H^n}$,
	and they are parametrized by the wavenumber $\alpha\in\R$ and the ``direction'' of the wave $u\in S^{n-1}$.
	
	\begin{remark}
		The analogy with the Euclidean case is perhaps more geometrically intuitive in the case $n=2$, once one moves to the disk model
		of the hyperbolic plane: in such a setting, $\e_2$ is rewritten as an imaginary exponential involving the distance between
		the horocycle through $x$ and $u\in S^1$ and the origin, the direction $u\in S^1$ being naturally identified with
		a point of the boundary of the Poincar\'e disk (a point at infinity of the hyperbolic plane).
		We refer to \cite[Introduction]{Helgason2000} for a thorough comparison between Euclidean and hyperbolic settings.
		We also observe that the analogy with the Euclidean case carries through when considering wave equations, justifying the ``wave'' terminology. In particular, solutions of the wave equation on $\H^n$,
		\begin{equation}\label{eq:waveequation}
		\pa{\de_t^2+\Delta_{\H^n}-\sigma^2}u(x)=0,
		\end{equation}
		(see \cite{Anker2012} for a discussion of this PDE) can be written as superpositions of waves $e^{it\alpha} \e_n(x,\alpha,u)$. 
		Notice that the wave operator in the previous display
		takes into account that the spectrum of $\Delta_{\H^n}$ begins at $\sigma^2$.
	\end{remark}
	\noindent
	Just as on $\R^n$, planar waves can be used to set up Fourier analysis on $\H^n$.
	
	\begin{proposition}\label{prop:fourier}\cite[Ssec. 2.11.4]{Hislop1994}
		Given $f\in C^\infty_c(\H^n)$, define its Fourier transform as
		\begin{equation}\label{eq:Fouriertransform}
		\F f(\alpha,\vartheta)= \int_{\H^n} \e_n(x,-\alpha,\vartheta)f(x)dm_n(x),
		\quad \alpha\in [0,\infty),\,\vartheta\in S^{n-1}.
		\end{equation}
		It holds (transform inversion)
		\begin{equation}\label{eq:Fourierinversion}
		f(x)=\int_0^\infty \int_{S^{n-1}} \F f(\alpha,\vartheta) \e_n(x,\alpha,\vartheta)
		\rho_n(\alpha) d\alpha d\varsigma_{n-1}(\vartheta).
		\end{equation}
		Moreover (Plancherel formula) $\F$ extends to an isometry
		\begin{equation*}
		\F:L^2(\H^n,m_n)\to L^2([0,\infty)\times S^{n-1}, \rho_n(\alpha)d\alpha d\varsigma_{n-1}),
		\end{equation*}
		whose inverse is given by (the extension of) \eqref{eq:Fourierinversion}.
	\end{proposition}
	\noindent
	Spherical functions can be regarded as spherical averages of waves $\e_n$,
	\begin{equation*}
	F_{n,\lambda}(d(x,(1,0,\dots,0)))=\frac1{\omega_{n-1}}\int_{S^{n-1}} \e_n(x,\alpha,\vartheta) d\varsigma_{n-1}(\vartheta),
	\end{equation*}
	(a special case of \autoref{eq:Fdef1}) 
	thus playing the role of Bessel functions in the Euclidean case --- see \eqref{eq:berrycov}. 
	
	\subsection{Hyperbolic Random Waves}\label{ssec:hyperbrw}
	In what follows we will consider both real-valued and complex-valued random fields;
	we refer to \cite[Chapter 6]{Hida1980} for a discussion of white noise analysis in the complex setting.
	
	\begin{remark}[Real and complex white noise] \label{r:whitenoise}
		Before stating the main result of the present Section, and for the reader's convenience, we recall the definition and basic properties
		of complex white noises, in the sense of \cite{Hida1980}. Fix a finite mesure space $(X, \mathcal{F}, \mu)$ and denote by $L^2(X,\mu;\R)$ and $L^2(X,\mu;\C)$, respectively, the associated $L^2$ spaces of real- and complex-valued functions. A (real) {\em white noise} on $(X, \mathcal{F}, \mu)$ (often called an {\em isonormal Gaussian process} with intensity $\mu$ --- see e.g. \cite[Chapter 2]{Nourdin2012}) is a centred real Gaussian family of the type $U=\{U(f) : f\in L^2(X,\mu;\R)\}$ such that  
		\begin{equation*}
		\expt{U(f)U(g)}=\int_X f g\,  d\mu,\quad f,g\in L^2(X,\mu;\R);
		\end{equation*}    
		the definition of $U$ is customarily extended to all $f\in L^2(X,\mu;\C)$ by setting $U(f) := U(\re(f)) +i U(\im(f))$.    
		A {\em complex white noise}
		$W$ on a finite measure space $(X,\mu)$ is a complex Gaussian family $W = \{W(f) : f\in L^2(X,\mu;\C)\}$ having the law of $U+ i V$, where $U,V$ are i.i.d. real white noises on $(X, \mathcal{F}, \mu)$, as defined above. The following computational rules can be easily checked: for all $f,g\in L^2(X,\mu;\C)$, one has that
		\begin{gather}\label{eq:cpxwncov}
		\expt{W(f)\obar{W(g)}}=\int_X f \obar{g} \, d\mu,\quad
		\expt{W(f)W(g)}=0,\\ \label{eq:cpxwnrere}
		\expt{\re[W(f)]\re[W(g)]}=\int_X [\re(f) \re(g)+\im(f) \im(g)] d\mu,\\ \label{eq:cpxwnreim}
		\expt{\re[W(f)]\im[W(g)]}=\int_X [\re(f) \im(g)-\im(f) \re(g)] d\mu,
		\end{gather}
		and the second and third equalities continue to hold when one switches the symbols `$\re$' and `$\im$' on both sides of each equation.
	\end{remark}

	The next proposition singles out a class of stationary random fields that can be regarded as {\em canonical Gaussian Laplace eigenfunctions} on $\H^n$ --- they will constitute our main object of study.

	\begin{proposition}\label{prop:randomwave}
		Fix $\alpha\in [0,\infty)$, and set $\lambda=\sigma^2+\alpha^2$, where the constant $\sigma^2$ is the same as in \eqref{e:parispectrum}.
		\begin{itemize}
			\item[\bf (1)] There exists a unique (in law) random field $u_\lambda:\H^n\to\R$ such that
			\begin{itemize}
				\item[(i)] $u_\lambda(x)$ is a Gaussian variable $N(0,1)$ for all $x\in\H^n$;
				\item[(ii)] the law of $u_\lambda$ is invariant under isometries of $\H^n$; 
				\item[(iii)] almost all samples of $u_\lambda$ are generalized $\lambda$-eigenfunction of $\Delta_{\H^n}$
				of class $C^\infty(\H^n)$.
			\end{itemize}
			The same conclusion holds for the complex version $u_\lambda^\C:\H^n\to\C$ if at Point {\rm (i)}
			one replaces the standard Gaussian random variable $N(0,1)$ with a standard complex Gaussian variable $N_\C(0,1)$.
			\item[\bf (2)] The Gaussian random field $u_\lambda$ is equivalently characterized 
			by its mean and covariance function
			\begin{equation}\label{eq:covstrut}
			\expt{u_\lambda(x)}=0, \quad \expt{u_\lambda(x)u_\lambda(y)}=F_\lambda(d(x,y)),
			\end{equation}
			for all $x,y\in\H^n$; samples of $u_\lambda$ are of class $C^\infty(\H^n)$.
			Moreover, $\frac1{\sqrt 2}\re u_\lambda^\C$ and $\frac1{\sqrt 2}\im u_\lambda^\C$
			are two independent identically distributed real Gaussian random fields with the same law as $u_\lambda$.
			\item[\bf (3)] We have the following representation: if $W$ is a complex white noise on $(S^{n-1}, \varsigma_{n-1}) $, 
			then $u_\lambda^\C$ has the same law as the stochastic integral
			\begin{equation}\label{eq:stochint}
			u_\lambda^\C(x) \sim \int_{S^{n-1}} \e_n(x,\alpha,\vartheta) W(d\vartheta).
			\end{equation}
		\end{itemize}
	\end{proposition}
	\begin{remark} The random fields $u_\lambda^\C$ are discussed -- with different notation and from a slightly different perspective -- in \cite[Section 5.3]{Cohen2012}, to which the reader is referred for further background material. For the rest of the paper, we will refer to $u_\lambda$ and $u_\lambda^\C$, respectively, as the real and complex {\em hyperbolic random wave} with eigenvalue $\lambda$. 
	\end{remark}

	\begin{proof}[Proof of Proposition \ref{prop:randomwave}]
		It is convenient to start by \emph{defining} $u_\lambda^\C(x)$ by means of \eqref{eq:stochint} 
		and show that it satisfies the properties put forward at Point {\bf (2)}. 
		This will show in particular that the real-valued function $(x,y)\mapsto F_{n,\lambda}(d(x,y))$ is positive definite for all $\lambda \in [\sigma^2,\infty)$,
		so that \eqref{eq:covstrut} uniquely identifies the law of a real Gaussian random field. To prove that the properties at Point {\bf (2)} are met by the random field in \eqref{eq:stochint}, we start by observing that the Gaussianity of the stochastic integral is trivial, and so is the fact that
		\begin{equation*}
		\expt{\int_{S^{n-1}} \e_n(x,\alpha,\vartheta) W(d\vartheta)}=0
		\end{equation*} 
		for all $x\in\H^n$ and $\alpha\geq 0$. As for the covariance,
		we deduce from \eqref{eq:cpxwncov} that
		\begin{multline*}
		\expt{\int_{S^{n-1}} \e_n(x,\alpha,\vartheta) W(d\vartheta)
			\overline{\int_{S^{n-1}} \e_n(y,\alpha,\vartheta) W(d\vartheta)}}\\
		=\int_{S^{n-1}} \e_n(x,\alpha,\vartheta)\e_n(y,-\alpha,\vartheta)d\varsigma_{n-1}(\vartheta),
		\end{multline*}
		which by \eqref{eq:Fdef1} equals $F_{n,\lambda}(d(z,w))$.
		Moreover, \eqref{eq:cpxwnrere} and \eqref{eq:cpxwnreim},
		combined with the fact that (by definition) $\re [\e_n(x,\alpha,\vartheta)]=\re [\e_n(x,-\alpha,\vartheta)]$ and $\im [\e_n(x,\alpha,\vartheta)]=-\im [\e_n(x,-\alpha,\vartheta)]$,
		show that $\frac1{\sqrt 2}\re u_\lambda^\C$ and $\frac1{\sqrt 2}\im u_\lambda^\C$
		are two i.i.d. centered real Gaussian random fields with covariance function
		$F_{n,\lambda}(d(\cdot,\cdot))$. This shows that \eqref{eq:stochint} satisfies the properties at Point {\bf (2)} (note that $u_\lambda$ and $u_\lambda^\C$ have paths of class $C^\infty$ because the covariance function of these fields
		is of class $C^\infty$: this implication is proved e.g. in \cite[Subsection A.9]{NS2016} for Gaussian fields on Euclidean spaces, and it is straightforwardly adapted to the hyperbolic setting after composition with a (smooth, global) chart of $\H^n$. The proof of the Theorem is concluded if we show the equivalence of \eqref{eq:covstrut}
		and of the properties listed at Point {\bf (1)}. Let us assume that $u_\lambda$ verifies the properties at Point {\bf (1)}. Then, by invariance under isometries (and since the isometry group of $\H^n$ acts transitively), the covariance function
		\begin{equation}
		C(z,w)=\expt{u_\lambda(z)u_\lambda(w)}=f(d(z,w))
		\end{equation}
		only depends on the distance $d(z,w)$. Since the samples of $u_\lambda$ are smooth generalized eigenfunctions,
		we then deduce that
		\begin{equation}
		\Delta_{\H^n}C(z,w)=\expt{\Delta_{\H^n}u_\lambda(z)u_\lambda(w)}
		=\lambda \expt{u_\lambda(z)u_\lambda(w)}=\lambda C(z,w),
		\end{equation}
		and from the discussion in Subsection \ref{ssec:spectral}, we conclude that $C(z,w)=F_\lambda(d(z,w))$. The proof that \eqref{eq:covstrut} implies the conditions at Point {\bf (1)} follows from similar arguments, both in the real- and complex-valued cases. \end{proof}
	
	To further the analogy with Berry's random waves on $\R^n$, one can also derive
	the random field $u_\lambda$ with a Central Limit result for a superposition
	of finitely many generalized eigenfunctions of $\Delta_{\H^n}$.
	
	\begin{proposition}
		Let $\alpha\in [0,\infty)$, $\lambda=\sigma^2+\alpha^2$ be fixed, 
		and consider two independent sequences of i.i.d. uniform random
		variables $\vartheta_1,\vartheta_2,\dots$ on $S^{n-1}$ and $\phi_1,\phi_2,\dots$ on $[0,2\pi]$.
		Define the following finite combination of hyperbolic waves
		\begin{equation}
		u^N_\lambda(z)=\frac1{\sqrt N} \sum_{j=1}^N e^{i\phi_j} \e_n(x,\alpha,\vartheta_j),
		\end{equation}
		to be regarded as a random element of $C^\infty(\H^n;\C)$.
		As $N\to \infty$, finite-dimensional distributions of $u^N_\lambda$ converge in law 
		to the ones of $u_\lambda^\C$.
	\end{proposition}
	
	A real analogue of the latter can be obtained by taking the real (or imaginary) part of all involved objects.
	Notice how, in sight of Subsection \ref{ssec:spectral}, 
	this result represents $u_\lambda$ as a stochastic superposition
	of single waves solving \eqref{eq:waveequation} with wavenumber $\alpha$.
	
	\begin{proof}
		For fixed $x\in \H^n$, $u^N_\lambda(z)$ can be regarded as the duality coupling between the smooth function
		$\e_n(x,\alpha,\cdot)\in C^\infty(S^{n-1};\C)$ and the generalized function
		\begin{equation*}
		\frac1{\sqrt N} \sum_{j=1}^N e^{i\phi_j}\delta_{b_j}(\cdot)\in C^\infty(S^{n-1};\C)^\ast.
		\end{equation*}
		Since generalized functions $e^{i\phi_j}\delta_{b_j}$ can be regarded as i.i.d. random elements of the Sobolev space
		$H^s(S^{n-1};\C)$ for $s<-n/2$, the Central Limit Theorem for i.i.d. variables in Hilbert spaces applies
		(\emph{cf.} \cite[10.1]{Ledoux1991}),
		and the sum in display converges in law to complex white noise $W$ on $S^{n-1}$.
		The thesis then follows by \autoref{prop:randomwave} considering couplings with $\e_n(x,\alpha,\cdot)$
		at finitely many distinct points $x$.
	\end{proof}
	
	\subsection{Curvature, Large Scale and Local Behavior of Random Waves}\label{ssec:curvature}
	
	As already discussed, the principal aim of the present paper is to characterize the fluctuations of integral functionals of the hyperbolic waves $\{u_\lambda\}$, as defined in the previous Subsection \ref{ssec:hyperbrw}, both as $\lambda\to \infty$ on a fixed domain ({\em high-frequency limit}), and for fixed $\lambda$ on expanding domains ({\em large domain limit}). Our main achievements on the matter are discussed in full detail in the forthcoming Section \ref{sec:intfunc}: in particular, our findings will show some remarkable discrepancies between the large domain behaviours of hyperbolic and Euclidean polyspectra.
	
	In order to develop some basic intuition on the relation between hyperbolic and Euclidean settings, in Proposition \ref{prop:localbehavior} we will characterize the local behaviour of hyperbolic random waves around a fixed point -- that we will encode in terms of the scaling limit of the associated pullback waves on tangent spaces. Some preliminary considerations are, however, in order.
	
	\subsubsection{Remarks on scaling limits}

	We start by pointing out a fundamental difference between the hyperbolic and Euclidean settings, that is: {\em in the hyperbolic framework} -- and differently from the Euclidean one -- {\it there is no direct relation linking high-frequency and large distance limits}. 
	
	To see this, fix $\lambda>0$  and recall the definition of the Euclidean random waves $\{v_\lambda\}$ introduced in \eqref{eq:berrycov}. Trivially, the fact that $C_{n,\lambda}(x,y)$ is a function of $\sqrt\lambda |x-y|$ makes it so that for Euclidean random waves it is equivalent to consider limits at high frequency (for a fixed distance) and at large distance (at a fixed frequency). 
	We will see that this is \emph{not} the case for (functionals of) hyperbolic waves.
	
	Indeed, the counterpart of scaling lengths on a Euclidean space is to consider a positive multiple of the metric tensor on a Riemannian manifold. Namely, if $M=(M,g)$ is a Riemannian manifold we set $M_R=(M,R^2g)$, $R>0$, a transformation that amounts to
	multiply all distances by $R$: if $x,y\in M$, $d_M(x,y)=r$, then $d_{M_R}(x,y)=Rr$. Under this transformation, eigenvalues of Laplace-Beltrami operator are scaled by a factor $1/R^2$.
	
	In the Euclidean case $M=\R^n$, if $\phi^R_\lambda(x,y)=\phi^R_\lambda(|x-y|)$ is the unique radial solution of
	\begin{equation*}
	\Delta_R \phi^R_\lambda(x,y)=\lambda \phi^R_\lambda(x,y), \quad \phi^R_\lambda(x,x)=1,
	\end{equation*}
	where $\Delta_R=\frac1{R^2}\Delta$ is the Laplace operator on $\R^n_R$, then
	\begin{equation}\label{eq:scalings}
	\phi^R_\lambda(|x-y|)=\phi^1_{R^2\lambda}(|x-y|)=C_{R^2\lambda}(x,y)=\phi^1_\lambda(R|x-y|),
	\end{equation}
	where the last equality is a consequence of the particular form of spherical functions on flat space.
	
	Consider now the hyperbolic case: a crucial difference is that $\R^n_R, R>0,$ are all isometric,
	whereas $\H^n_R$ has sectional curvature $-1/R^2$. Looking at spherical functions, 
	if $\psi^R_\lambda(x,y)=\psi^R_\lambda(d(x,y))$ is the unique radial solution of
	\begin{equation*}
	\Delta_{\H^n_R} \psi^R_\lambda(x,y) = \lambda \psi^R_\lambda(x,y), \quad \psi^R_\lambda(x,x)=1,
	\end{equation*}
	then the first equation of \eqref{eq:scalings} still holds,
	\begin{equation*}
	\psi^R_\lambda(d(x,y))=\psi^1_{R^2\lambda}(d(x,y))
	\end{equation*}
	it being a general fact (notice that $d(x,y)$ is the distance of $\H^n=\H^n_1$, not the rescaled one).
	However, in sight of the last display and the previous paragraphs, we can write
	\begin{align*}
	\psi^R_\lambda(x,y)&=F_{n,R^2\lambda}(d(x,y))\\
	&=	C_n \int_0^\pi \pa{\cosh d(x,y)-\sinh d(x,y) \cos\theta}^{-\sigma+iR\sqrt{\lambda-\sigma^2/R^2}} (\sin \theta)^{n-2} d\theta, 
	\end{align*}
	in terms of the function $F_{n,\lambda}$ defined in \eqref{eq:Fdef2}, which makes it clear that
	\begin{equation*}
	\psi^R_\lambda(d(x,y))=\psi^1_{R^2\lambda}(d(x,y))\neq \psi^1_\lambda(Rd(x,y)),
	\end{equation*}
	marking the difference with the Euclidean case.
	
	\subsubsection{A local scaling limit result} In the light of the above discussion, 
	a natural question is whether the local behavior of the hyperbolic waves $u_\lambda$
	around a given point resembles that of Berry's model at high frequencies.
	This turns out to be the case --- at least from the standpoint of covariance functions.
	Since the two models are defined on different manifolds, such a statement
	is made precise by comparing the planar random wave on $\R^n$ with covariance function
	as in \eqref{eq:berrycov} and frequency $\lambda=1$, 
	and a properly rescaled pullback of $u_\lambda$ to the tangent space
	(at a given point $x\in \H^n$) given by the exponential map.
	
	\begin{proposition}\label{prop:localbehavior}
		Let $\alpha\geq0$ be fixed, set $\lambda=\sigma^2+\alpha^2$ and fix $x\in\H^n$.
		Consider the covariance function 
		of the pullback random wave $u_\lambda(\exp_x(\cdot/\sqrt\lambda))$ on $T_x\H^n\simeq \R^n$,
		\begin{equation*}
		C_{n,\lambda}^H(u,v)=F_\lambda\pa{d\pa{\exp_x \frac{v}{\sqrt\lambda},\exp_x \frac{v'}{\sqrt\lambda}}}, \quad v,v'\in\R^n.
		\end{equation*}
		Let $r_\lambda=o(\sqrt \lambda)$ as $\lambda\to \infty$. Then,
		recalling that $C_{n,\lambda}$ denotes the covariance of Berry's Euclidean model \eqref{eq:berrycov}, one has that
		\begin{equation*}
		\sup_{v,v'\in\R^2: |v|,|v'|< r_\lambda} \abs{C_{n,\lambda}^H(v,v')- C_{n,1}(v,v')}=o(1),
		\quad \lambda\to\infty.
		\end{equation*}
	\end{proposition}
	
	\medskip
	
	We refer to the forthcoming Subsection \ref{ssec:rasymptotic} for a proof. The content of Proposition \ref{prop:localbehavior} mirrors the characterisation of the local behaviour of {\em monochromatic random waves} on compact Riemannian manifolds --- see e.g. \cite{Canzani2020, DNPR} and the references therein --- with the important difference that, in the high-frequency limit, the covariance function of monochromatic random waves locally converge to $C_{n,1}$ in any $C^k$ topology.  Although we did not check the details, it is plausible that a similar result could be achieved also in our setting, by exploiting an asymptotic expansion similar to the one used in our proof of Proposition \ref{prop:localbehavior}. Plainly, if this result was available, one could directly implement coupling techniques analogous to the ones developed in \cite{DNPR}, and deduce {\em small scale CLTs} for geometric functionals of hyperbolic waves. 
	
	\section{Integral Functionals: Wiener Chaos Expansion and CLTs}\label{sec:intfunc}
	
	It is a standard fact (see e.g. \cite[Chapter 2]{Nourdin2012}) that square-integrable random variables of the form $G(u_\lambda(x))$, where $u_\lambda(x)\sim N(0,1)$,
	can be decomposed into a converging series of random elements with the form $H_q(u_\lambda(x))$, where $\{H_q : q\geq 0\}$ denotes the sequence of uni-variate {\it Hermite polynomials}  (see Section \ref{sss:chaos} for more details). The aim of the present Section is thus to characterize the asymptotic behaviour of integral functionals of the form
	\begin{equation}\label{e:gul}
	\G_R(u_\lambda):=\int_{B_R} G(u_\lambda(x))dm_n(x)
	\end{equation}
	(in both regimes $\lambda\to\infty$ and $R\to \infty$, $B_R\subset \H^n$ being a hyperbolic ball of radius $R$) by first studying the normal approximation of the \emph{polyspectra}
	\begin{equation}\label{eq:polyspectra}
	h^{n,q}_{R,\lambda}\coloneqq \int_{B_R} H_q(u_\lambda(x))dm_n(x), \quad q\geq 1.
	\end{equation}
	We will see in Section \ref{ssec:leray} that our results allow one to deal even with cases
	in which $G$ is not a proper function, but a distribution on $\R$.
	
	\begin{remark}\label{r:wass} We will characterize the fluctuations of random variables by using a specific probabilistic distance. More precisely, given two integrable random variables $X,Y$, we define the {\it 1-Wasserstein distance} between the laws of $X$ and $Y$ as 
		$$
		{\bf W}_1(X,Y) := \sup_h \Big|\,  \expt{h(X)} - \expt{h(Y)} \Big|,
		$$
		where the supremum runs over all 1-Lipschitz functions $h : \R\to \R$. See e.g. \cite[Appendix C]{Nourdin2012}, and the references therein, for some basic results about the distance ${\bf W}_1$. The following facts can be easily checked:

		\medskip
		
		\begin{itemize}
			\item[\bf (a)] if ${\bf W}_1(X_k,Y) \to 0$, as $k\to \infty$, then $X_k$ converges in distribution to $Y$;
			
			\medskip
			
			\item[\bf (b)] For any $X,Y$ integrable and $b>0$,
			\begin{equation}\label{e:wassscale}
			{\bf W}_1(X,Y) = b \cdot {\bf W}_1(X/b, Y/b).
			\end{equation}

			\medskip
			
			\item[\bf (c)] let $X_k$, $k\geq 1$, be a centered and square-integrable random variable such that $ a\, c^2_k \leq \var{X_k} \leq b\, c^2_k $, for some strictly positive sequence $\{c^2_k\}$ and constants $a,b>0$, and let $N_k$ denote a centered Gaussian random variable with variance $\var{X_k} / c^2_k$; then if 
			\begin{equation}
			\label{e:conwass1}
			{\bf W}_1(X_k / c_k , N_k) \to 0,
			\end{equation}
			for every subsequence $k(n)\to \infty$ there exists a sub-subsequence $k({n'})$ such that $\var{X_{k({n'})}} / c^2_{k(n')}$ converges and
			\begin{equation}
			\label{e:conwass2}
			\frac{X_{k(n')}}{c_{k(n')}} \stackrel{{\rm Law}}{\Longrightarrow} N(0, c^2),
			\end{equation}
			where $c^2 := \lim_{n'} c^{-2}_{k(n')}\var{X_{k({n'})}} >0$. 
		\end{itemize}
		The content of Point {\bf (c)} amplifies the relevance of the forthcoming Theorem \ref{thm:fourthmoment}.
	\end{remark}
	
	\subsection{Some elements of Gaussian analysis} \label{ssec:polyspectra} 
	
	\subsubsection{Representation of hyperbolic waves} Let the notation and assumptions of the previous Sections prevail. Then, for all $n\geq 1$ and $\lambda\in \left[\pa{\frac{n-1}{2}}^2,\infty\right)$ (see \eqref{e:parispectrum}), one has that there exists a finite measure space $(X,\F,\mu)$, a real white noise $W$ on $(X,\mu)$ (as defined in Remark \ref{r:whitenoise}) and an integral kernel $K_{n,\lambda}:\H^n\times X \to \R$ such that the random field
	\begin{equation}\label{e:wnrep}
	\H^n \ni x\mapsto I_1(K_{n,\lambda}(x,\cdot)):= \int_{X} K_{n,\lambda}(x,y) W(dy),
	\end{equation}
	has the same law as $\{u_\lambda(x) : x\in \H^n\}$. A representation of the type \eqref{e:wnrep} can be deduced from Point {\bf (3)} of Proposition \ref{prop:randomwave} (in which case, $X = S^{n-1}$ and $\mu$ is the uniform measure). In general, we stress that (i) the representation \eqref{e:wnrep} is not unique, (ii) the validity of \eqref{e:wnrep} is a standard consequence of the fact that each random field $\{u_\lambda\}$ is separable, and (iii) the space $(X,\F,\mu)$ and the random measure $W$ can be chosen to be the same for each $n$. 
	
	\smallskip
	
	{\it Without loss of generality, from now on we will assume that representation \eqref{e:wnrep} is in order, and that the real white noise $W$ on $(X,\mu)$ is defined on a probability space $(\Omega, \mathcal{G}, \mathbb{P})$ in such a way that $W$ is the same for each $n$ and $\mathcal{G}$ is the canonical completion of $\sigma(W)$.}
	
	\subsubsection{Wiener chaos}\label{sss:chaos} The following basic elements of Gaussian stochastic analysis will be used for the rest of the paper --- see \cite{Janson1997, Nourdin2012} for a full discussion. For every $f\in L^2(X,\F,\mu)$, the stochastic integral $I_1(f) := \int_X f \, dW$ is well-defined, and the class $\{I_1(f) : f\in L^2(X,\F,\mu)\}$ is a centered Gaussian family (known as the {\it first Wiener chaos} of $W$) with covariance given by the relation: for all $f,g\in L^2(X,\F,\mu):= L^2(\mu)$, $\expt{I_1(f)I_1(g)}=\brak{f,g}_{L^2(\mu)} $. 
	
	\smallskip
	
	It is a standard fact that the space $L^2(\PP):=L^2(\Omega,\G,\PP)$ admits the {\it Wiener chaotic decomposition}
	\begin{equation}\label{e:wienerchaos}
	L^2(\PP)=\bigoplus_{q=0}^\infty H^{:q:}, \quad
	H^{:0:}:=\R,\,\,\, H^{:q:}:=\set{I^q(f):\, f\in L^2_{\rm sym}(\mu^q)}, \, \,\, q\geq 1,
	\end{equation}
	where the symbol $L^2_{\rm sym}(\mu^q)$ stands for the Hilbert subspace of $L^2(X^q, \mathcal{F}^{\otimes q}, \mu^q) : = L^2(\mu^q)$ composed of (the equivalence classes of) those kernels $f$ that are $\mu^q$-almost everywhere symmetric, and $I^q(f)$ denotes the $q$-fold stochastic integral of $f$ with respect to $W$. For $q\geq 1$, $H^{:q:}$ is called the $q$th {\it Wiener chaos} of $W$, and one has the isometric relation: $\mathbb{E}[I_q(f) I_p(g)] ={\bf 1}_{p=q} \, q! \brak{f,g}_{L^2(\mu^q)}$, valid for all $p,q\geq 1$, and all $f\in L^2_{\rm sym} (\mu^q)$ and $g\in L^2_{\rm sym} (\mu^p)$. 
	
	\smallskip
	
	We will often exploit the fact that, for all $f\in L^2(X,\F,\mu)$, one has that $I_q(f^{\otimes q})=H_q(I_1(f))$, where $H_q$ is the $q$-th (probabilistic) Hermite polynomial on the real-line; see e.g. \cite[Theorem 2.7.7]{Nourdin2012}. We recall that the collection $\{H_q : q=0,1,...\}$ of Hermite polynomials coincides with the coefficients of the exponential generating function
	\begin{equation*}
	e^{st-t^2/2}=\sum_{q=0}^{\infty} H_q(s) \frac{t^q}{q!},\quad t,s\in\R.\footnote{The first few polynomials are: $H_0(x) = 1$, $H_1(x) = x$, $H_2(x) = x^2-1$, $H_3(x) = x^3-3x$, $H_4(x) = x^4-6x^2+3$, and so on.}
	\end{equation*}
	It is well-known that $\set{H_q/\sqrt{q!}}_{q\geq 0}$ is an orthonormal basis of $L^2(\R,\phi(s)ds)$, where $\phi(s) =(2\pi)^{-1/2}  e^{-s^2/2} $ is the standard Gaussian density. 
	
	\smallskip
	
	According to the previous conventions and discussion, for all $x\in \H^n$, $u_\lambda (x) = I_1(K_{n,\lambda}(x,\cdot))$
	belongs to the first Wiener chaos $H^{:1:}$ and, consequently, the random variable
	\begin{equation*}
	H_q(u_\lambda(x))=I^q\pa{K_{n,\lambda}(x,\cdot)^{\otimes q}}
	\end{equation*}
	is an element of $H^{:q:}$, as defined in \eqref{e:wienerchaos}.
	Given $G\in L^2(\R)$ we deduce 
	--- e.g. by dominated convergence and Jensen's inequality ---  
	that the chaos expansion of the integral functional $\G_R(u_\lambda)$ defined in \eqref{e:gul} is given by
	\begin{gather}\label{eq:functionalG}
	\G_R(u_\lambda)=\sum_{q=0}^\infty h^{n,q}_{R,\lambda} \frac1{q!}\int_\R G(t)H_q(t)\phi(t)dt, \\
	\label{eq:polyspectraint}
	h^{n,q}_{R,\lambda}= I_q\pa{\int_{B_R} K_{n,\lambda}(x,\cdot)^{\otimes q}dm_n(x)},
	\end{gather}
	where the series \eqref{eq:functionalG} converges in $L^2(\PP)$ and its $q$th summand coincides with the projection of $G_R(u_\lambda)$ onto $H^{:q:}$; in particular, for $q\geq 1$, the \emph{polyspectrum} $h^{n,q}_{R,\lambda}$ is an element of $H^{:q:}$ and (by applying e.g. a stochastic Fubini argument) it is easily seen to coincide with \eqref{eq:polyspectra} above.
	The smallest $q\geq 0$ for which the coefficient $\int_\R G(t)H_q(t)\phi(t)dt$
	does not vanish is called the \emph{Hermite rank} of $G$ (and, by extension, of the functional $\G_R(u_\lambda)$). As recalled in the Introduction, the notion of Hermite rank plays a pivotal role in the asymptotic theory of Gaussian-subordinated random fields on Euclidean spaces, see e.g. \cite[Chapter 7]{Nourdin2012}.

	\subsubsection{The Wiener chaos approach to CLTs}\label{ss:wienerchaosstrategy} Mantaining the notation and assumptions of the previous Section, we will now state three results allowing one to prove CLTs by using Wiener chaos expansions: these statements are the core of the ``Wiener chaos approach'' advertised in the title of the paper.
	
	\smallskip
	
	Given $q\geq 1$ and $f,g\in L^2_{\rm sym}(\mu^{q})$, and $r=0,...,q$, the $r$-{\it contraction} of $f$ and $g$ is defined as $f\otimes_0 g = f\otimes g$ and, for $r=1,...,q$,
	\begin{multline*}
	(f\otimes_r g)(x_1,\dots x_{2q-2r})\\=\int_{X^r} f(x_1,\dots x_{q-r}, z_1,\dots,z_r)
	g(x_{q-r+1},\dots, x_{2q-2r}, z_1,\dots,z_r)d\mu^{r}(z_1,\dots,z_r),
	\end{multline*}
	in such a way that $f\otimes_r g \in L^2(\mu^{2q-2r})$, with the convention $L^2(\mu^0) := \R$.  
	
	\smallskip
	
	As made clear in the next statement, which is a direct consequence of \cite[Theorem 5.2.7 and Theorem 6.3.1]{Nourdin2012}, contractions can be used in order to quantitatively assess the distance to Gaussian within a fixed Wiener chaos.
	
	\begin{theorem}\label{t:wasschaos} Fix $q\geq 2$ and let $f\in L^2_{\rm sym}(\mu^{q})$ be such that $q!\| f\|^2_{L^2(\mu^q)} = \expt{I_q(f)^2} = \sigma^2>0$. Then, there exists a combinatorial constant $\Gamma(q) >0$, uniquely depending on $q$, such that
		\begin{equation}\label{e:wasschaos}
		{\bf W}_1(I_q(f), N(0,\sigma^2)) \leq \frac{\Gamma(q)}{\sigma}  \, \max_{r=1,...,q-1} \| f\otimes_r f \|_{L^2(\mu^{2(q-r)})}.
		\end{equation}
	\end{theorem}
	
	Combining \eqref{e:wassscale} with \cite[Proposition 3.7]{NPTRF}, one can use contractions to derive effective bounds on the normal approximation of random variables living in a finite sum of Wiener chaoses. 
	\begin{proposition}\label{p:sumclt} Let $Q\geq 1$ and let 
		$$
		F = \sum_{q=1}^Q I_q(g_q), \quad g_q\in L^2_{\rm sym}(\mu^{q}),
		$$
		be such that $\expt{F^2} = \sum_{q=1}^Q q! \|g_q\|^2_{L^2(\mu^q)} = \sigma^2>0$. Then, there exists a combinatorial constant $\Gamma_1(Q)$, uniquely depending on $Q$, such that
		\begin{equation}
		{\bf W}_1(F, N(0,\sigma^2))\leq \frac{\Gamma_1(Q)}{\sigma} \cdot \max_{q,r} \|g_q \otimes_r g_q\|_{L^{2(q-r)}}.
		\end{equation} 
		where the maximum runs over all $q=1,...,Q$ and all $r=1,...,q-1$.
	\end{proposition}
	
	\smallskip
	
	Finally, in order to derive CLTs in the context of generic nonlinear functionals of hyperbolic waves, we will need the following general result.
	
	\begin{theorem}\label{thm:fourthmoment}
		Under the above assumptions and notation, consider a sequence 
		of square-integrable random variables with the form
		\begin{equation*}
		F_j=\sum_{q=1}^\infty I_q(f_{j,q}), \quad f_{j,q}\in L^2_{\rm sym}(\mu^{q}),\quad j\geq 1,
		\end{equation*}
		and write $c^2_{j,q} := q! \| f_{j,q} \|^2_{L^2(\mu^q)}$. Consider a sequence $c^2(j) \to \infty$ with the following property: there exist $T,T_0\subset \mathbb{N}$ such that $T\cap T_0 =\emptyset$, $T\neq \emptyset$ and $T\cup T_0 = \mathbb{N}$, as well as finite constants $0<b_1<b_2$ and $\sigma^2_q\geq 0$, verifying
		\begin{enumerate}
			\item $\sum_{q=1}^\infty \sigma_q^2:= \Sigma_0<\infty$;
			\item for all $q\in T$, $\sigma^2_q>0$ and
			$$
			b_1\, \sigma^2_q \,  c^2(j) \leq c^2_{j,q} \leq b_2 \, \sigma^2_q \,  c^2(j), \quad j\geq 1; 
			$$
			\item for all $q\in T_0$, $c^2_{q,j}\leq b_2 \, c^2(j)\,  \sigma^2_q$;
			\item for fixed $q\geq 1$ and all $r=1,\dots,q-1$,	
			\begin{equation*}
			\lim_{j\to\infty}c^{-2}(j) \norm{f_{j,q}\otimes_r f_{j,q}}_{L^2(\mu^{2 (q-r) })}=0.
			\end{equation*}
		\end{enumerate}
		Then, as $j\to \infty$, 
		\begin{equation}
		\label{e:conwass3}
		{\rm Var}(F_j) \simeq c^2(j)\mbox{\quad and \quad} {\bf W}_1\left(\frac{F_j}{ c(j)} , N(0, \gamma^2(j)) \right) \longrightarrow 0,
		\end{equation}
		where $\gamma^2(j) :=  {\rm Var}(F_j) /c^2(j)$. 	
		
		\begin{remark} The discussion around formulae \eqref{e:conwass1}--\eqref{e:conwass2} above illustrates the significance of the second asymptotic relation in \eqref{e:conwass3}.
		\end{remark}

		\begin{proof}[Proof of Theorem \ref{thm:fourthmoment}] One has that
			$$
			b_1\, \sum_{q\in T} \sigma^2_q \, \, c^2(j) \leq  {\rm Var}(F_j)\leq b_2\, \Sigma_0 \, c^2(j),
			$$
			from which one deduces that ${\rm Var}(F_j) \simeq c^2(j)$. Now write $N_j := N(0, \gamma^2(j))$, fix $Q$ such that $T\cap \{1,...,Q\} \neq \emptyset$ (such a $Q$ exists because $T$ is non empty), and observe that -- by applying twice the triangle inequality -- 
			$$
			{\bf W}_1(F_j / c(j) , N_j)\leq 2 \sqrt{b_2\!\!\sum_{q\geq Q+1} \sigma^2_q}  + {\bf W}_1\left(\frac{1}{c(j)} \sum_{q=1}^Q I_q(f_{j,q}), N(0, \sigma^2(j,Q) ) \right)
			$$
			where $\sigma^2(j,Q) := c(j)^{-2} \sum_{q=1}^Q c^2_{j,q}\geq b_1 \sigma^2_{q_0}>0$, and $q_0$ is any element of $T\cap \{1,...,Q\}$. It follows from Proposition \ref{p:sumclt} that
			\begin{eqnarray*}
				{\bf W}_1(F_j / c(j) , N_j)\!\!\!\!&\leq&\!\!\!\! 2\sqrt{b_2\!\!\sum_{q\geq Q+1} \sigma^2_q} +\frac{\Gamma_1(Q)}{{b^{1/2}_1} \sigma_{q_0}} \max_{q,r} \frac{\norm{f_{j,q}\otimes_r f_{j,q}}_{L^2(\mu^{2 (q-r) })}}{c^2(j)} ,\\
				&:=& A(Q) + B(Q,j).
			\end{eqnarray*}
			where the maximum runs over all $q\leq Q$ and all $r=1,...,q-1$. Since $A(Q)\to 0$ (as $Q\to \infty$) and $B(Q,j)\to 0$ (as $j\to \infty$ for fixed $Q$), the conclusion follows.
		\end{proof}	
	\end{theorem}
	\noindent

	\subsection{CLTs for integral functionals}\label{ssec:clts}
	Our aim is now to apply \autoref{thm:fourthmoment} to the integral functionals $\G_R(u_\lambda)$ (as defined in \eqref{e:gul}), both as $\lambda\to \infty$ and as $R\to \infty$. In view of \eqref{eq:functionalG}--\eqref{eq:polyspectraint}, this task requires one to assess both the variances and the contraction norms associated with the polyspectra $h^{n,q}_{R,\lambda}$, once these random elements are represented as multiple stochastic integrals --- with respect to the white noise $W$ on $(X, \mu)$ featured in \eqref{e:wnrep} --- of kernels of the type $$\int_{B_R} K_{n,\lambda}(x,\cdot )^{\otimes q}dm_n(x) \in L^2_{\rm sym} (\mu^q).$$ 
	
	One fundamental fact (explaining why, for our purposes, the precise choice of the white noise $W$ and measure space $(X, \mathcal{F}, \mu)$ is immaterial) is that both the variances and the contraction norms associated with the polyspectra  $h^{n,q}_{R,\lambda}$ uniquely depend on the covariance function $\cov(u_\lambda(x),u_\lambda(y))=F_{n,\lambda}(d(x,y))$. To see this, we observe that, by a standard argument based e.g. on \cite[Proposition 2.2.1]{Nourdin2012}, for all $q\geq 1$ one has that
	\begin{eqnarray}\label{e:varh}
	\var(h^{n,q}_{R,\lambda})
	&=&\expt{\int_{B_R}\int_{B_R} H_q(u_\lambda(x))H_q(u_\lambda(y)) dm_n(x) dm_n(y)}\\ \notag
	&=& q! \int_{B_R}\int_{B_R} F_{n,\lambda}(d(x,y))^q dm_n(x) dm_n(y).
	\end{eqnarray}
	Analogously, a direct computation (based on the use of Fubini theorem as well as on the representation \eqref{e:wnrep} and the isometric properties of real white noises stated in Remark \ref{r:whitenoise}) yields the equation
	\begin{multline}\label{eq:contractionsK}
	\norm{\pa{\int_{B_R} K_{n,\lambda}(x,\cdot)^{\otimes q}dm_n(x)}
		\otimes_r \pa{\int_{B_R} K_{n,\lambda}(y,\cdot)^{\otimes q}dm_n(y)}}_{L^2(\mu^{\otimes 2 (q-r) })}^2\\
	=\int_{B_R^4}
	F_{n,\lambda}(d(x,y))^r F_{n,\lambda}(d(y,z))^{q-r}\\ \cdot F_{n,\lambda}(d(z,w))^r F_{n,\lambda}(d(w,x))^{q-r}
	dm_n^{\otimes 4}(x,y,z,w).
	\end{multline}
	
	\medskip

	Appropriate tools for estimating expressions such as \eqref{e:varh}--\eqref{eq:contractionsK} are developed in Section \ref{sec:moments}. As an application of these techniques, we will prove the following statement, which is one of the main achievements of the paper.
	\begin{theorem}\label{thm:cltHermite}
		Let $n\geq 2$ and $q\geq 1$. We have the following asymptotic relations for ${q!}^{-1} \, \var(h^{n,q}_{R,\lambda})$:
		\begin{center}
			\begin{tabular}{l|l|l}
				${q!}^{-1}\var(h^{n,q}_{R,\lambda})$ & $\lambda\to\infty$, $R$ fixed & $R\to\infty$, $\lambda$ fixed \\
				\hline
				$q=1$& $\lesssim_{n,R} \lambda^{-\sigma-1}$& $\lesssim_{n,\lambda} m_n(B_R)$\\
				$q=2$& $\simeq_{n,R} \lambda^{-\sigma}$& $\simeq_{n,\lambda} R\cdot m_n(B_R)$\\
				even $q\geq 4$, except $n=2,q=4$& $\simeq_{n,R} \lambda^{-\sigma-1/2}$ & $\simeq_{n,\lambda} m_n(B_R)$\\
				$n=2,q=4$& $\simeq_{n,R} \log(\lambda)/\lambda$ & $\simeq_{n,\lambda} m_n(B_R)$ \\
				odd $q\geq 3$, except $n=q=3$& $\lesssim_{n,R} \lambda^{-\sigma-1/2}$ & $\lesssim_{n,\lambda} m_n(B_R)$\\
				$n=q=3$& $\lesssim_{n,R} \lambda^{-3/2}\log(\lambda)$ & $\lesssim_{n,\lambda} m_n(B_R)$\\
			\end{tabular}
		\end{center}
		\noindent
		In both limiting regimes considered above, for all $1\leq r\leq q-1$, the squared contraction norms \eqref{eq:contractionsK} 
		are of order $o\left(( q')!^{-2} \var(h^{n,q'}_{R,\lambda})^2\right)$ for all even $q'\leq q$, where the implicit constants depend on $n$ and $R$ (as $\lambda\to \infty$) and on $n$ and $\lambda$ (as $R\to \infty$). As a consequence, one has that, for all $n\geq1$ and all $q$ even,
		\begin{equation}\label{e:cltpoly}
		\frac{h^{n,q}_{R,\lambda}}{\var(h^{n,q}_{R,\lambda})^{1/2}}\stackrel{{\rm Law}}{\Longrightarrow} N(0, 1),
		\end{equation}
		both as $\lambda \to\infty$, for $R$ fixed, and as $R\to \infty$, for $\lambda$ fixed.
	\end{theorem}
	
	The first part of the statement (concerning variances) is proved in \autoref{lem:momentlambda} and \autoref{lem:momentr};
	the second part (on contraction norms) is proved in \autoref{lem:contractionslambda} and \autoref{lem:contractionsr}. The CLT \eqref{e:cltpoly} is a direct consequence of \eqref{e:wasschaos}.
	
	\begin{remark}\label{r:hr2} By inspection of the proofs of the four Lemmas \ref{lem:momentlambda}, \ref{lem:momentr}, \ref{lem:contractionslambda} and \ref{lem:contractionsr}, one can extrapolate some more precise information about the rate of convergence of squared contraction norms. For instance, in the case $q=2$, one has that, writing $X(R,\lambda, n)$ for the ratio
		$$
		\frac{\int_{B_R^4}
			F_{n,\lambda}(d(x,y)) F_{n,\lambda}(d(y,z))\\ \cdot F_{n,\lambda}(d(z,w)) F_{n,\lambda}(d(w,x))
			dm_n^{ 4}(x,y,z,w)}{\var(h^{n,2}_{R,\lambda})^2}, $$
		for all $n\geq 1$,
		\begin{equation}\label{e:q2} 
		X(R,\lambda, n) \,\, \begin{cases}
		\lesssim_{n,R}\lambda^{-\sigma-2}, \quad \lambda\to\infty,\\
		\lesssim_{n,\lambda} \frac{1}{R}, \, \quad R\to\infty;\\
		\end{cases}
		\end{equation}
		Finally, applying Theorem \ref{t:wasschaos} yields the following estimate on the speed of convergence in \eqref{e:cltpoly}:
		\begin{equation}\label{e:x}
		{\bf W}_1\left( \frac{h^{n,2}_{R,\lambda}}{\var(h^{n,2}_{R,\lambda})^{1/2}}, N(0,1)\right) \leq c\cdot  X(R,\lambda, n)^{1/2},
		\end{equation}
		where $c>0$ is some absolute constant. For general $q>2$ even, similar estimates can be deduced from the proof of Lemma \ref{lem:contractionslambda} and the statement of Lemma \ref{lem:contractionsr}. Again by virtue of Theorem \ref{t:wasschaos}, such estimates yield bounds on the speed of convergence (in the 1-Wasserstein distance) in the CLT \eqref{e:cltpoly}.
	\end{remark}

	\begin{remark}
		We recall that, as $R\to \infty$, one has that $m_n(B_R)\simeq_n e^{2\sigma R}$. 
		In the previous statement, we preferred to use expressions involving $m_n(B_R)$ in order to better streamline the comparison with the Euclidean case, as done in the next remark.
	\end{remark}
	
	\begin{remark}[Comparison with the Euclidean case] For every $n\geq 2$, consider the Euclidean random wave with energy $\lambda>0$ featured in \eqref{eq:berrycov}, and define the Euclidean polyspectrum $h_{e;R,\lambda}^{n,q}$ according to \eqref{eq:polyspectra}, by replacing $u_\lambda$ with $v_\lambda$, and by considering that $B_R$ is the Euclidean ball of radius $R$ centered at the origin, and $m_n$ is the Lebesgue measure on $\R^n$. Then, the following table of asymptotic relations (that we extrapolated from \cite[Theorem V.1.1]{MNPHD} -- see also \cite{Nourdin2019, Peccati2020, Notarnicola2022} and \cite{Dalmao2022}, respectively, for the cases $n=2$ and $n=3$) is valid, with $\sigma$ defined as in \eqref{e:parispectrum}:
		
		\smallskip
		
		\begin{center}
			\begin{tabular}{l|l|l}
				$q!^{-1} \var(h^{n,q}_{e;R,\lambda})$ & $\lambda\to\infty$, $R$ fixed & $R\to\infty$, $\lambda$ fixed \\
				\hline
				$q=1$& $\lesssim_{n,R} \lambda^{-\sigma-1}$& $\lesssim_{n,\lambda} m_n(B_R)/R$\\
				$q=2$& $\simeq_{n,R} \lambda^{-\sigma}$& $\simeq_{n,\lambda} R\cdot m_n(B_R)$\\
				even $q\geq 2$, except $n=2,q=4$& $\simeq_{n,R} \lambda^{-\sigma-1/2}$ & $\simeq_{n,\lambda} m_n(B_R)$\\
				$n=2,q=4$& $\simeq_{R} \log(\lambda)/\lambda$ & $\simeq_{\lambda}\log(R)\cdot m_n(B_R)$ \\
				odd $q\geq 3$, except $n=q=3$& $\lesssim_{n,R} \lambda^{-\sigma-1/2}$ & $\lesssim_{n,\lambda} m_n(B_R)$\\
				$n=q=3$& $\lesssim_{R} \lambda^{-3/2}\log(\lambda)$ & $\lesssim_{\lambda} \log(R)\cdot m_n(B_R)$\\
			\end{tabular}
		\end{center}
		
		\smallskip
		
		\noindent Moreover, estimates on the contraction norms analogous to the ones in the statement of Theorem \ref{thm:cltHermite} hold. As discussed in the Introduction, the most remarkable difference between this table and the one in Theorem \ref{thm:cltHermite} is that, in the case $n=2, \, q=4$ and for $R\to\infty$, the variance of $h^{n,q}_{e;R,\lambda}$ does not display any logarithmic correction.
	\end{remark}
	
	\begin{remark}[Comparison with random spherical harmonics]
		Laplace-Beltrami operator $\Delta_{S^n}$ on the sphere $S^n$ has a discrete spectrum,
		and an orthonormal basis of $L^2(S^n)$ that diagonalizes $\Delta_{S^n}$ is provided by spherical harmonics
		\begin{gather*}
		\Delta_{S^n} Y_{\ell,m;n}=\ell(\ell+n-1)Y_{\ell,m;n}, \quad\ell\in\N,\,m=1,2\dots,d_{\ell;n},\\
		d_{\ell;n}=\frac{2\ell+n-1}{\ell}\binom{\ell+n-2}{\ell-1},
		\end{gather*}
		$\ell$ playing a role similar to the one of $\alpha$ in hyperbolic waves, that is $\ell$ is proportional to the
		square root of the eigenvalue $\lambda=\ell(\ell+n-1)$ and $\ell\in\N$ parametrizes the spectrum.
		The index $m$ on the other hand parametrizes a single eigenspace, $d_{\ell;n}$ denoting its dimension.
		Random spherical harmonics are defined by
		\begin{equation*}
		T_\lambda(x)=\sum_{m=1}^{d_{\ell;n}} a_{\ell,m} Y_{\ell,m;n},\quad x\in S^n,\,\lambda=\ell(\ell+n-1),
		\end{equation*}
		with $a_{\ell,m}$ being i.i.d. Gaussian variables with
		$\expt{a_{\ell,m}a_{\ell,m'}}=\delta_{m=m'}\omega_{n}/d_{\ell;n}$,
		and we can consider polyspectra $h_{\operatorname{sph},\lambda}^{n,q}=\int_{S^n} H_q(T_\lambda(x))d\varsigma_n(x)$. We cannot consider a large-domain limiting regime on $S^n$: in this case we report variance asymptotics only
		in the high-frequency regime, matching the ones of Euclidean and hyperbolic cases.
		\begin{center}
			\begin{tabular}{l|l}
				$q!^{-1} \var(h_{\operatorname{sph},\lambda}^{n,q})$ & $\lambda\to\infty$\\
				\hline
				$q=1$& $=0$\\
				$q=2$& $\simeq_{n} \lambda^{-\sigma}$\\
				even $q\geq 2$, except $n=2,q=4$& $\simeq_{n} \lambda^{-\sigma-1/2}$ \\
				$n=2,q=4$& $\simeq \log \lambda/\lambda$ \\
				odd $q\geq 3$ & $\lesssim_{n} \lambda^{-\sigma-1/2}$
			\end{tabular}
		\end{center}
		We refer to \cite{Marinucci2015} for the latter results.
		Notice that the case $n=q=3$ is not included as an exception, because in fact
		$h_{\operatorname{sph},\lambda}^{n,q}$ identically vanish if
		$n,q$ are both odd, just as in the case $q=1$ (any $n$) in the table.
		This is an artifact of having chosen the whole $S^n$ as integration domain:
		symmetries of spherical harmonics come into play producing cancellations.
		In dimension $n=2$ the study was extended to polyspectra over spherical caps
		in \cite{Todino2019}, obtaining the following: for $\Omega\subset S^2$ 
		a spherical cap subtended by a solid angle,
		\begin{center}
			\begin{tabular}{l|l}
				$q!^{-1} \var(\int_\Omega H_q(T_\lambda(x))d\varsigma_2(x))$ & $\lambda\to\infty$\\
				\hline
				$q=1$& $\lesssim \lambda^{-3/2}$\\
				$q=2$& $\simeq_{n} \lambda^{-1/2}$\\
				$q=4$& $\simeq \log \lambda/\lambda$ \\
				even $q\geq 6$& $\simeq_{n} \lambda^{-1}$ \\
				odd $q\geq 3$ & $\lesssim_{n} \lambda^{-1}$
			\end{tabular}
		\end{center}
		(constants in the estimates are independent of $\Omega$).
		We refer to the series of works \cite{Marinucci2011,Marinucci2011a,Marinucci2011book,Marinucci2014,Marinucci2015,Marinucci2021,Rossi2019}
		for a complete overview of the Wiener chaos approach to
		integral functionals of random spherical harmonics.
	\end{remark}

	\begin{remark}\label{r:odd}
		When $q$ is odd, Theorem \ref{thm:cltHermite} only yields upper bounds for $\var(h^{n,q}_{R,\lambda})$ (in both regimes):
		obtaining a lower bound in this case is indeed complicated by the fact that such a variance displays oscillations that are arbitrarily close to zero. For this reason, in the forthcoming Proposition \ref{prop:cltgeneral} we are not able establish CLTs for functionals of odd Hermite rank. This point, together with an explanation of the fact that the cases $n=2,q=4$ and $n=q=3$ of the table feature a logarithmic correction in the high-frequency regime, will be fully discussed in Section \ref{sec:moments} below.
	\end{remark}
	
	\begin{remark}\label{rmk:unifinq}
		In the table appearing in Theorem \ref{thm:cltHermite}, the asymptotic relations for $q\geq 5$ have no explicit dependence on $q$ (the latter is indeed not featured among the subscripts). In fact, for all $n\geq 2$,
		\begin{multline}\label{eq:unifinq}
		\sup_{q\geq 6} q!^{-1} \var(h^{n,q}_{R,\lambda})
		\leq \sup_{q\geq 6} \int_{B_R}\int_{B_R} | F_{n,\lambda}(d(x,y))| ^q dm_n(x) dm_n(y)\\
		\leq \int_{B_R}\int_{B_R} |F_{n,\lambda}(d(x,y))|^6 dm_n(x) dm_n(y) 
		\\=6!^{-1} \var(h^{n,6}_{R,\lambda})
		\begin{cases}
		\lesssim_{n,R}\lambda^{-\sigma-1/2}, \quad \lambda\to\infty,\\
		\lesssim_{n,\lambda} m_n(B_R), \, \quad R\to\infty,\\
		\end{cases}
		\end{multline}
		so that, together with upper bounds for polyspectra with smaller $q$ (if needed),
		one can deduce asymptotic upper bounds for variances of polyspectra that are \emph{uniform in} $q$.
	\end{remark}
	
	Combining \autoref{thm:cltHermite} and \autoref{thm:fourthmoment}, one deduces the following general result for integral functionals of hyperbolic random waves (see once again the discussion around \eqref{e:conwass1}--\eqref{e:conwass2} in order to appreciate the significance of the conclusion \eqref{e:dwintegral}).
	
	\begin{proposition}\label{prop:cltgeneral}
		Consider the random variable $\G_R(u_\lambda)$, as defined in \eqref{eq:functionalG} for some $G\in L^2(\R, \,  \phi(s)ds)$ 
		with even Hermite rank $q_0\geq 2$. Write $v^2(q_0; R; \lambda) :=  \var(h^{n,q_0}_{R,\lambda})$.
		\begin{itemize}
			\item[{\bf (1)}] Suppose $q_0 \geq 4$. Then, the following asymptotic relations hold:
			\begin{eqnarray}
			&& \var \G_R(u_{\lambda})\simeq_{n,R,G}  v^2(q_0; R; \lambda), \quad \mbox{as } \lambda\to \infty,\\ 
			&& \var \G_R(u_{\lambda})\simeq_{n,\lambda,G}  v^2(q_0; R; \lambda), \quad\mbox{as } R \to \infty.
			\end{eqnarray} 
			Moreover, denoting by $N(R ; \lambda)$ a centered Gaussian random variable with the same variance of $\widetilde{\G_R}(u_{\lambda}) :=\G_R(u_{\lambda})/v(q_0; R; \lambda)$, one has that, both as $\lambda\to\infty$ and as $R\to\infty$,
			\begin{equation}\label{e:dwintegral}
			{\bf W}_1\left( \widetilde{\G_R}(u_{\lambda}), N(R ; \lambda) \right) \longrightarrow 0.
			\end{equation}
			\item[{\bf (2)}] If $q_0=2$, then, writing $a(G) := \frac1{2}\int_\R G(t)H_2(t)\phi(t)dt \neq 0$, one has that
			\begin{equation}\label{e:hr2v}
			\frac{\var \G_R(u_{\lambda})}{a^2(G)\cdot v^2(2; R; \lambda)} \longrightarrow 1, \mbox{ both as } \lambda\to \infty \mbox{ and as }R\to\infty.
			\end{equation}
			Moreover, one has the following bounds: if $R\to \infty$,
			\begin{equation}\label{e:v1}
			{\bf W}_1\left(\frac{\G_R(u_{\lambda})}{  | a(G) | \cdot v(2; R; \lambda)} , N(0,1) \right)     \lesssim_{n,\lambda,G} \frac{1}{R^{1/2}}; 
			\end{equation}
			if $\lambda\to\infty$,
			\begin{equation}\label{e:v2}
			{\bf W}_1\left(\frac{\G_R(u_{\lambda}) }{  | a(G) | \cdot v(2; R; \lambda)} , N(0,1) \right)  \,\, \begin{cases}
			\lesssim_{R,G}\sqrt{ \frac{\log \lambda }{\lambda } } ,\,  \quad n=2,\\
			\lesssim_{R,G} \sqrt{ \frac{\log \lambda }{\lambda^{1/2} }  }, \, \quad n=3,\\
			\lesssim_{n,R,G}\frac{1}{\lambda^{1/4}}, \, \, \,\quad n\geq 4.
			\end{cases}    
			\end{equation}
			\item[{\bf (3)}] If $q_0=4$ and $n=2$, writing $b(G) := \frac1{24}\int_\R G(t)H_4(t)\phi(t)dt \neq 0$,  it holds moreover that,
			as $\lambda\to\infty$,
			\begin{gather}
			\frac{\var \G_R(u_{\lambda})}{b^2(G)\cdot v^2(4; R; \lambda)} \longrightarrow 1,\\
			{\bf W}_1\pa{\frac{\G_R(u_{\lambda}) }{  | b(G) | \cdot v(4; R; \lambda)} , N(0,1)}\lesssim_R \frac1{\log \lambda}.
			\end{gather}
		\end{itemize}
	\end{proposition}
	
	In the case $q_0=2$, any $n\geq 2$ and both regimes, and $q_0=4,n=2$, large frequency,
	we are able to obtain quantitative statements since
	the first non-vanishing Hermite projection dominates
	the chaos expansion in the asymptotic regime.
	In the other discussed cases this does not happen,
	and a finer control of the considered functional is required.
	
	\begin{proof}{}
		[Proof of {\bf (1)}] Since the arguments needed to deal with the case $R\to \infty$ are analogous, we will only discuss the limiting regime $\lambda\to\infty$.
		Let $q_0\geq 2$ be the Hermite rank of $\G_R(u_\lambda)$.
		Since we are assuming that $q_0$ is even, the table in the statement of \autoref{thm:cltHermite}, combined with \autoref{rmk:unifinq}, implies the uniform estimate
		\begin{equation*}
		\sup_{q> q_0} q!^{-1} \var(h^{n,q}_{R,\lambda})= O_{n,R}(\var(h^{n,q_0}_{R,\lambda})).
		\end{equation*}
		The conclusion now follows by selecting a sequence $\lambda_j\to\infty$ and by applying \autoref{thm:fourthmoment} to the following special case:
		\begin{itemize}
			\item[(a)] $c^2(j) = v^2(q_0; R; \lambda_j)$;
			\item[(b)] $T_0 = \{q_0\}$;
			\item[(c)] $$ f_{j,q}=\int_{B_R} K_{n,\lambda_j}(x,\cdot)^{\otimes q}dm_n(x)\cdot  q!^{-1} \int_\R G(s)H_q(s)\phi(s)ds;$$
			\item[(d)] $\sigma^2_q = \frac{1}{q!}\left(  \int_\R G(s)H_q(s)\phi(s)ds\right)^2$;
			\item[(e)] $\Sigma_0 = \expt{G(N)^2}$, where $N$ is a standard Gaussian random variable.
		\end{itemize}
		\smallskip
		[Proof of {\bf (2)}] By virtue of \eqref{eq:functionalG}, the projection of ${\G_R}(u_{\lambda})$ onto the second Wiener chaos $H^{:2:}$ is given by $Z:= a(G)\cdot h^{n,2}_{R,\lambda}$. Using the triangle inequality, one has that
		\begin{eqnarray*}
			{\bf W}_1\left(\frac{\G_R(u_{\lambda}) }{  | a(G) | \cdot v(2; R; \lambda)} , N \right) &\leq&  {\bf W}_1\left(\frac{Z }{  | a(G) | \cdot v(2; R; \lambda)} , N \right) \\ &&+\frac{ {\rm Var} ( \G_R(u_{\lambda})-Z )^{1/2} }{   | a(G) | \cdot v(2; R; \lambda)} := A+B.
		\end{eqnarray*}
		A direct application of \eqref{e:x} yields that $A\leq c X(R,\lambda,n)^{1/2}$, where $c$ is some absolute combinatorial constant. On the other hand, one can directly bound $B$ by exploiting the orthogonality of distinct Wiener chaoses together with the asymptotic relations put forward in Theorem \ref{thm:cltHermite}. Combining these estimates with \eqref{e:q2} yields the desired result.
		The [Proof of {\bf (3)}] follows along the same lines,
		the asymptotics of the ratio between contractions and variance being directly deduced from the ones collected in the proof of \autoref{lem:contractionslambda}.\qedhere
	\end{proof}
	
	\subsection{First application: excursion volumes at non-zero levels}\label{ssec:excursion} As before, we write $\phi$ to indicate the standard Gaussian density, and also introduce the notation $\Phi(t) := \int_{-\infty}^t \phi(s)ds$. Our aim in this Section is to use Proposition \ref{prop:cltgeneral} in order to study the asymptotic behavior of random variables of the type
	\begin{equation}\label{e:phit}
	\Phi_{R,\lambda}(t)\coloneqq \int_{B_R} \one_{(-\infty,t]}(u_\lambda(x))dm_n(x) = m_n\{ x\in B_R : u_\lambda(x)\leq t\},
	\end{equation}
	defined for all $t\in \R$ and all $R>0,\lambda\geq \sigma^2$ (where $\sigma^2$ is defined, as usual, in \eqref{e:parispectrum}). 
	
	\smallskip
	
	We will first derive the chaotic expansion of $\Phi_{R,\lambda}(t)$. A direct application of Tonelli's theorem shows immediately that
	\begin{equation*}
	\expt{\Phi_{R,\lambda}(t)}=m_n(B_R) \Phi(t).
	\end{equation*}
	Moreover, for all $t\in \R$ the Hermite polynomial expansion of the function ${\bf 1}_{(-\infty,t]}$, in the space $L^2(\R, \phi(s)ds)$, is given by
	\begin{equation}\label{e:hei}
	{\bf 1}_{(-\infty,t]}(s)=\sum_{q=0}^\infty \frac1{q!}\psi_q(t) H_q(s),
	\end{equation}
	where $\psi_q(t)=-H_{q-1}(t)\phi(t),\,\, q\geq 1$, and $\psi_0(t)=\Phi(t)$.\footnote{
		The functions $\psi_q$ are the \emph{Hermite functions}, forming an
		orthonormal basis of $L^2(\R,dx)$ that diagonalizes the Fourier transform operator on the real line.}
	An immediate, easy consequence of \eqref{e:hei} is also that, for $N$ a standard Gaussian random variable,
	\begin{equation*}
	\var(\chi_{(-\infty,t]}(N))=\Phi(t)(1-\Phi(t))= \sum_{q=1}^\infty \frac{\psi_q(t)^2}{q!}.
	\end{equation*}
	We observe that, since the functions $H_{q-1}$ are odd for even $q\geq 2$, one has that $\psi_q(0)=0$ for all $q\geq 2$ even: in particular, this implies that the Hermite expansion of ${\bf 1}_{(-\infty,t]}$ in the critical case $t=0$ uniquely involves Hermite polynomials of odd order (in such a way that this special case cannot be dealt with by using the results of the present paper). Reasoning as in the previous Section, we also deduce that, for all $t\in\R$, the Wiener chaos expansion of $\Phi_{R,\lambda}(t)$ is given by
	\begin{equation}\label{eq:areachaos}
	\Phi_{R,\lambda}(t)= \sum_{q=0}^\infty A_q(t) h^{n,q}_{R,\lambda}, 
	\quad A_0(t)=\Phi(t), \quad A_q(t)=\psi_q(t)/q!, \, q\geq 1.
	\end{equation}
	where the polyspectra $h^{n,q}_{R,\lambda}$ are defined as in \eqref{eq:polyspectra}.
	
	\smallskip
	
	Combining the above discussion \autoref{prop:cltgeneral} we deduce the following statement:
	
	\begin{proposition}\label{prop:cltexcursion}
		Let the above assumptions and notation prevail and fix $n\geq 1$ and $t\neq 0$. For fixed $R>0$ and $\lambda\to\infty$,
		\begin{equation*}
		\var(\Phi_{R,\lambda}(t)) \simeq_{n,R}  t^2\phi(t)^2\cdot \lambda^{-\sigma},
		\end{equation*}
		whereas, for fixed $\lambda\geq \sigma^2$ and $R\to\infty$,
		\begin{equation*}
		\var(\Phi_{R,\lambda}(t)) \simeq_{n,\lambda} t^2\phi(t)^2 \cdot R \, m_n(B_R).
		\end{equation*}
		Moreover, writing $\widetilde{\Phi}_{R,\lambda}(t) :=\big( \Phi_{R,\lambda}(t)- m_n(B_R)\Phi(t)\big)/ \var(\Phi_{R,\lambda}(t))^{1/2}$, one has the following explicit estimates: if $R\to \infty$, then $${\bf W}_1\left(\widetilde{\Phi}_{R,\lambda}(t)  , N(0,1) \right)\to 0$$ at a speed upper-bounded by the right-hand side of \eqref{e:v1}; if $\lambda\to \infty$, then $${\bf W}_1\left(\widetilde{\Phi}_{R,\lambda}(t)  , N(0,1) \right)\to 0$$ with an upper bound given by the right-hand side of \eqref{e:v2}.
	\end{proposition}
	
	\begin{proof}
		The projection on $H^{:1:}$ of $\Phi_{R,\lambda}(t)$ is $P(t, \lambda, R) :=\phi(t)h^{n,1}_{R,\lambda}$; according to Theorem \ref{thm:cltHermite}, one has that $\var{P(t,\lambda, R) /\var(h^{n,2}_{R,\lambda})} \lesssim_{n,R} \lambda^{-1}$, as $\lambda\to \infty$ and $\var{P(t,\lambda, R) /\var(h^{n,2}_{R,\lambda})} \lesssim_{n,\lambda} R^{-2}$, as $R\to \infty$. The result now follows from an application of Proposition \ref{prop:cltgeneral} to the function $G(x) = {\bf 1}_{(-\infty, t]}(x) - \Phi(t) - \phi(t)\, x$ (which has Hermite rank equal to 2), and from the fact that the projection on $H^{:2:}$ of $\Phi_{R,\lambda}(t)$ is $\frac{t}{2}\phi(t)h^{n,2}_{R,\lambda}$.
	\end{proof}
	
	\begin{remark}
		As already recalled, in the nodal case $t=0$ all projections on even Wiener chaoses in \eqref{eq:areachaos} vanish:
		since we only have {\it upper} asymptotic bounds on odd polyspectra, in that case
		we are not able to deduce a CLT for the functional.
		In fact, this is the same issue faced in the context of 2d random spherical harmonics by \cite{Marinucci2011}, to which we refer also for an interesting chaining argument produce a CLT for the normalized functional $\Phi_{R,\lambda}(t)-m_n(B_R)\Phi(t)$ \emph{as a stochastic process indexed by} $t\in\R\setminus\set{0}$.
		Limit theorems at $t=0$ were later derived for random waves on $S^2$
		in \cite{Marinucci2011a} by means of an \emph{ad hoc} argument, that we are not yet able to replicate in our setting. 
	\end{remark}

	In the next Section, we present a direct analysis of the so-called ``Leray measures'' of nodal sets of hyperbolic waves.

	\subsection{Second application: Leray measures of nodal sets}\label{ssec:leray}
	According to our previous discussion, the samples of $u_\lambda$ are smooth functions on $\H^n$, and classical results on stationary Gaussian random fields ensure moreover that, for all $t$, the level set $u_\lambda^{-1}(t)$ is a submanifold of codimension 1 
	almost surely (\emph{cf.} Bulinskaya's Lemma, \cite[Proposition 6.12]{Azais2009}).
	One can thus consider generalized functions on $\H^n$ supported on $u_\lambda^{-1}(t)$,
	the fundamental one being defined as
	\begin{equation}\label{e:lrsmooth}
	\brak{\varphi, \delta_{u_\lambda^{-1}(t)}}= \int_{u_\lambda^{-1}(t)} \varphi(x) d\omega(x),
	\end{equation}
	where $d\omega$ indicates integration with respect to the volume form on $u_\lambda^{-1}(t)$
	induced by the metric of $\H^n$, and with
	brackets denoting duality coupling with smooth functions $\varphi\in C^\infty_c$ (we refer e.g. to \cite[III.1]{Gelfand1964}
	for the classical theory of distributions in this setting).
	
	The aim of this Section is to study the high-frequency and large domain asymptotic behavior of random variables $L_{R,\lambda}$ that are formally obtained from \eqref{e:lrsmooth} by taking $t=0$ and $\varphi (x) =\one_{B_R}(x)$.
	More precisely, for $\lambda\in \left[\pa{\frac{n-1}{2}}^2,\infty\right)$ and $R>0$ we define the {\it Leray measure} of the nodal set of $u_\lambda$ restricted to $B_R$ (also called the {\it occupation density at zero} of $u_\lambda$ on $B_R$) to be the quantity
	\begin{equation}\label{e:trueleray}
	L_{R,\lambda}
	:=\lim_{\epsilon\to 0} \frac1{2\epsilon} m_n\pa{\set{x\in B_R: |u_\lambda(x)|\leq \epsilon}}
	=:\lim_{\epsilon\to 0} L_{R,\lambda}^\epsilon
	\end{equation}
	whenever such a limit is well-defined in the sense of convergence in probability. Heuristically, the Leray measure $L_{R,\lambda}$ is the rescaled volume of the subset of $B_R$ in which $u_\lambda$ takes values that are infinitesimally close to zero --- see the classical reference \cite{GH80} for a general discussion of occupation densities, as well as \cite{Oravecz2008, Peccati2018} for similar studies of the Leray measures associated with arithmetic random waves.

	Should the limit \eqref{e:trueleray} exist in $L^2(\PP)$, one could derive the chaos expansion of $L_{R,\lambda}$ from that of $L_{R,\lambda}^\epsilon$, which is easily seen to be the following:
	\begin{equation*}
	L_{R,\lambda}^\epsilon 
	=\frac1{2\epsilon} \int_{B_R} \one_{[-\epsilon,\epsilon]}(u_\lambda(x))dm_n
	= \frac{\Phi_{R,\lambda}(\epsilon)-\Phi_{R,\lambda}(-\epsilon)}{2\epsilon}
	= \sum_{q=0}^\infty B_q^\epsilon h^{n,q}_{R,\lambda},
	\end{equation*}
	where $h^{n,q}_{R,\lambda}$ is defined in \eqref{eq:polyspectra} and, in the notation of \autoref{prop:cltexcursion},
	\begin{equation*}
	B_q^\epsilon=\frac{A_q(\epsilon)-A_q(-\epsilon)}{2\epsilon},
	\quad \mbox{with  } B_q:=\lim_{\epsilon\to 0} B_q^\epsilon=A_q'(0)=\frac{H_q(0)\phi(0)}{q!}.
	\end{equation*}
	The last asymptotic relation yields that, if \eqref{e:trueleray} holds in $L^2(\PP)$, then the chaos expansion of $L_{R,\lambda}$ is obtained from that of $L_{R,\lambda}^\epsilon$ by replacing each $B^\epsilon_q$ with $B_q$. Here are some additional (useful) remarks on the coefficients $B_q$:
	\begin{itemize}
		\item[--] one can regard the sequence $(B_q)_{q\geq 1}$ as given by the coefficients of a \emph{formal} Hermite decomposition
		\begin{equation*}
		\delta_{0}(Z)=\sum_{q\geq 0}B_q H_q(Z), \quad Z\sim N(0,1),
		\quad B_q=\brak{\delta_0, H_q/q!}_{L^2(\R,\phi(s)ds)},
		\end{equation*}
		where Dirac's delta is regarded as a generalized function on $\R$;
		\item[--] $B_q=0$ for all odd $q\geq 1$, whereas for $\ell \in \N$ one has that
		\begin{equation}\label{eq:Bq}
		B_{2\ell}=\frac{H_{2\ell}(0)}{\sqrt{2\pi (2\ell)!}}, \quad H_{2\ell}(0)=(-1)^\ell(2 \ell-1)!!,
		\end{equation}
		the latter being a consequence of the definition of $H_q$;
		\item[--] the following asymptotic relation is derived from Stirling's formula and \eqref{eq:Bq}:
		\begin{equation*}
		B_{2\ell}^2\simeq \frac1{\sqrt\ell}, \quad \ell\to\infty.
		\end{equation*}
	\end{itemize}
	In order to establish the existence in $L^2(\PP)$ of the limit \eqref{e:trueleray}, we will rely on the following result, which we state and prove in a general setting because of its independent interest.
	
	\begin{proposition}\label{prop:generalleray}
		Let $(Z, \mathscr{Z}, \mu)$ be a finite measure space, 
		and consider a real-valued centered Gaussian field $X(z)$,
		defined on a probability space $(\Omega,\F,\PP)$,
		indexed by $z\in Z$ with $X(z)\sim N(0,1)$ for all $z\in Z$ 
		and covariance function $R(z,z')=\cov(X(z),X(z'))$ (which therefore takes values in $[-1,1]$). For every measurable $C\subset Z\times Z$, one has that
		\begin{equation}\label{e:leray1}
		\sum_{\ell=0}^\infty B^2_{2\ell} \int_{C}  R(z,z')^{2\ell} d\mu^{2}(z,z')
		= \frac{1}{2\pi} \int_{C} \frac{1}{\sqrt{1-R(z,z')^2 }} d\mu^{ 2}(z,z').
		\end{equation}
		Moreover, the limit
		\begin{equation}\label{e:intdelta}
		\mathcal{R}=\int_Z \delta_0(X(z)) d\mu(z) 
		:=\lim_{\epsilon\to 0}\frac1{2\epsilon} \int_Z \one_{[-\epsilon, \epsilon]}(X(z))  d\mu(z),
		\end{equation}
		exists in $L^2(\PP)$ if and only if either side of \eqref{e:leray1} is finite when $C = Z\times Z$.
		In this case, one has that
		\begin{equation}\label{e:leraymean}
		\expt{\mathcal{R}}=\frac{\mu(Z)}{ \sqrt{2\pi}},\quad
		\expt{\mathcal{R}^2}=\frac{1}{2\pi} \int_{Z^2} \frac{1}{\sqrt{1-R(z,z')^2 }} d\mu^2(z,z').
		\end{equation}
	\end{proposition}
	
	\begin{remark}
		A by-product of Proposition \ref{prop:generalleray} is that,
		if $\mathcal{R}$ is well-defined as a limit in $L^2(\PP)$, the set of those $(z,y)\in Z^2$ such that $R(z,y) = \pm 1$
		is necessarily $\mu^2$-negligible.
	\end{remark}
	
	\begin{proof}[Proof of \autoref{prop:generalleray}]
		From \eqref{eq:Bq}, it follows that
		\begin{equation}\label{eq:powerseries}
		\sum_{\ell=0}^\infty B_{2\ell}^2 x^{2\ell} = 
		\frac{1}{2\pi} \sum_{\ell=0}^\infty\binom{2\ell}{\ell} \frac{x^{2\ell}}{4^\ell} 
		= \frac{1}{2\pi} \frac{1}{\sqrt{1-x^2}}, \quad x\in (-1,1).
		\end{equation}
		This directly implies \eqref{e:leray1} by monotone convergence.
		To prove the second part of the statement, 
		for every $\epsilon >0$ denote by $\mathcal{R}(\epsilon)$ the argument of the limit on the right-hand side of \eqref{e:intdelta}:
		$\mathcal{R}(\epsilon)$ is a square-integrable variable and its chaos decomposition is given as before by
		\begin{equation}\label{eq:chaosReps}
		\mathcal{R}(\epsilon)=\sum_{q=0}^\infty B^\epsilon_q \int_Z H_q(X(z))d\mu(z).
		\end{equation}
		We recall that $B^\epsilon_q = B_q=0$ for all $\epsilon >0$ and odd $q$. If either side of \eqref{e:leray1} is finite for $C=Z\times Z$, 
		recalling the elementary relation
		\begin{equation*}
		\expt{H_q(X(z))H_{q'}(X(z'))}= q!R(z,z')^q\cdot \one_{q=q'}, 
		\end{equation*}
		we deduce that the series
		\begin{equation*}
		\mathcal{X} =  \sum_{\ell = 0}^\infty B^2_{2\ell} \int_Z  H_{2\ell}(X(z)) \mu(dz)
		\end{equation*}
		converges in $L^2(\PP)$, and
		\begin{equation*}
		\expt{(\mathcal{R}(\epsilon) -\mathcal{X})^2}
		=\sum_{\ell = 0}^\infty \pa{B^\epsilon_{2\ell}-B_{2\ell}}^2
		\int_Z  \int_Z R(z,y)^{2l} \mu(dz)\mu(dy).
		\end{equation*}
		Since $B^\epsilon_q\xrightarrow{\epsilon\to 0}B_q$
		one can apply the well-known inequality for Hermite polynomials (see \cite[22.14.16]{Abramowitz1965}),
		\begin{equation*}
		\abs{H_{2\ell-1}(x)}\leq \frac{x e^{x^2/4}(2\ell)!}{2^\ell \ell!}, \quad \ell\geq 1,
		\end{equation*}
		from which one infers the estimates
		\begin{equation*}
		\abs{B^\epsilon_{2\ell}}\leq \frac{\phi(\epsilon) e^{\epsilon^2/4}}{2^\ell \ell!}
		\leq \frac{\phi(0)}{2^\ell \ell!}=
		\phi(0)\frac{(2\ell-1)!!}{(2\ell)!}=\abs{B_{2\ell}}, \quad l\geq 1,
		\end{equation*} 
		so that
		\begin{equation*}
		\pa{B^\epsilon_{2\ell}-B_{2\ell}}^2\leq 2 B_{2\ell}^2.
		\end{equation*}
		By dominated convergence, this yields that $\expt{(\mathcal{R}(\epsilon) -\mathcal{X})^2}\to 0$.
		This last relation implies that $\mathcal{R}$ is well-defined in $L^2(\PP)$ and that $\mathcal{R} = \mathcal{X}$. Conversely, if $\mathcal{R}$ is well-defined in $L^2(\PP)$, then its projection on the $q$-th Wiener chaos
		is obtained from the one of $\mathcal{R}(\epsilon)$, as given in \eqref{eq:chaosReps}, by taking the limit $\epsilon\to 0$
		(because of the continuity of the projection map from $L^2(\PP)$ to a closed subspace).
		The conclusions in \eqref{e:leraymean} now follow immediately by combining the relation $B_0=\frac1{\sqrt{2\pi}}$ with \eqref{e:leray1}.
	\end{proof}
	
	The next statement contains the main results of the Section.
	
	\begin{proposition}\label{prop:leray}
		The functional $L_{R,\lambda}$ is well-defined as the $L^2(\PP)$-limit 
		of $L_{R,\lambda}^\epsilon$ (as $\epsilon \to 0$; see \eqref{e:trueleray}) and its chaos expansion is given by
		\begin{equation*}
		L_{R,\lambda} = \sum_{\ell=1}^\infty B_{2\ell} h^{n,2\ell}_{R,\lambda}
		=\frac{m_n(B_R)}{\sqrt{2\pi}}-\frac1{2\sqrt{2\pi}}h^{n,2}_{R,\lambda}+\dots,
		\end{equation*}
		where the series converges in $L^2(\PP)$. The following asymptotic relations hold:
		\begin{equation*}
		\var(L_{R,\lambda})
		\begin{cases}
		\simeq_{n,R}\lambda^{-\sigma}, &\lambda\to \infty,\\
		\simeq_{n,\lambda} R\cdot m_n(B_R), &R\to \infty.
		\end{cases}
		\end{equation*}
		Moreover, setting $\widetilde{L}_{R,\lambda} :=\left(L_{R,\lambda} -\frac{m_n(B_R)}{\sqrt{2\pi}}\right)\var(L_{R,\lambda})^{-1/2}$, one has that $\widetilde{L}_{R,\lambda}$ converges in distribution towards a standard Gaussian random variable, both as
		$\lambda\to \infty$, with $R>0$ fixed, and as $R\to\infty$, with $\lambda\geq\sigma^2$ fixed.
	\end{proposition}
	
	\begin{proof} We start by recalling that, according to the table in the statement of \autoref{thm:cltHermite} one has that
		\begin{equation*}
		\var(h^{n,2}_{R,\lambda})
		\begin{cases}
		\simeq_{n,R}\lambda^{-\sigma}, &\lambda\to \infty,\\
		\simeq_{n,\lambda} R\cdot m_n(B_R), &R\to \infty.
		\end{cases}
		\end{equation*}
		Also, \autoref{cor:sqrtlambda} and \autoref{lem:sqrtR} below imply the following bounds:
		\begin{equation*}
		\abs{\int_{B_R^2}
			\frac{dm_n^{\otimes 2}(x,y)}{\sqrt{1-F_{n,\lambda}(d(x,y))^2}}
			-m_n(B_R)^2- \frac12 \var h^{n,2}_{R,\lambda}}
		\begin{cases}
		\lesssim_{n,R}\lambda^{-\sigma-1/2}, &\lambda\to \infty,\\
		\lesssim_{n,\lambda} m_n(B_R), &R\to \infty.
		\end{cases}
		\end{equation*}
		Combining these asymptotic relations, one infers that:
		\begin{itemize}
			\item[--] the random variable $L_{R,\lambda}$ is well-defined in $L^2(\PP)$ (as a direct application of \autoref{prop:generalleray} in the case $Z = B_R$, $\mu = m_n$ and $X = u_\lambda$);
			\item[--] denoting $P[R,\lambda\, ;\, \geq 3]$ the projection of $L_{R,\lambda}$ onto the direct sum $\bigoplus_{q\geq 3} H^{:q:}$, one has that
			$$
			\var(P[R,\lambda\, ;\,  \geq 3]) = o(\var(h^{n,2}_{R,\lambda})),
			$$
			both as $\lambda\to \infty$, with $R>0$ fixed, and as $R\to\infty$, with $\lambda\geq\sigma^2$ fixed. 
		\end{itemize}
		The conclusion now follows immediately from the second part of the statement of \autoref{thm:cltHermite}.
	\end{proof}

	\section{Covariance Functions and their Moments}\label{sec:moments}
	
	The arguments in the previous Section are based on asymptotic estimates of multiple integrals involving
	the covariance function $F_{n,\lambda}(d(x,y))$ of $u_\lambda$, such as the moments
	\begin{equation}\label{eq:momentscovariance}
	C_{R,\lambda}^{n,q}=\var(h^{n,q}_{R,\lambda})=
	\int_{B_R}\int_{B_R} F_{n,\lambda}(d(x,y))^q dm_n(x)dm_n(y),\quad q\geq 0,
	\end{equation}
	and 4-fold integrals appearing in \eqref{eq:contractionsK} as representations of kernel contractions
	in \autoref{thm:fourthmoment}.
	In order to obtain bounds on these integrals we need precise estimates on $F_{n,\lambda}$ itself,
	which we derive in the next paragraph. The proof of \autoref{prop:localbehavior} also requires such estimates,
	so we report it at the end of Subsection \ref{ssec:rasymptotic}, before moving to the main technical arguments of the paper
	in the remainder of the Section.
	
	\subsection{Approximating Covariance Functions} \label{ssec:rasymptotic}
	The following statement collects approximations of $F_{n,\lambda}$ we will employ
	to derive estimates on the double integrals $C_{R,\lambda}^{n,q}$ defined in \eqref{eq:momentscovariance}.
	
	\begin{lemma}\label{lem:Fasymp}
		Let $n\geq 2$, $\alpha\in\R^\ast$, $\lambda=\sigma^2+\alpha^2$, $r\geq 0$.
		It holds
		\begin{enumerate}
			\item (uniform bound)
			\begin{equation}\label{eq:Funif}
			\abs{F_{n,\lambda}(r)}\lesssim_n e^{-\sigma r},
			\end{equation}
			uniformly in $r\geq 0$; 
			moreover, for any $r_0>0$, uniformly in $r\geq r_0$,
			\begin{equation}\label{eq:Fremainder}
			\abs{\sinh(r)^{\sigma} F_{n,\lambda}(r)- \re\bra{c_n(\alpha) \sinh(r)^{i\alpha }}}
			\lesssim_n |\alpha|^{-2}\sinh(r)^{-2},
			\end{equation}
			where
			\begin{equation}\label{eq:hcfunc}
			c_n(\alpha)=\frac{2^{2\sigma-1}\Gamma(i\alpha)\Gamma(\sigma+1/2)}{\sqrt\pi \Gamma(\sigma+i\alpha)},
			\end{equation}
			is bounded for $\alpha\in\R^\ast$ away from zero, and $c_n(\alpha)\simeq_n |\alpha|^{-\sigma}$
			as $|\alpha|\to\infty$;
			\item (decay in $\alpha$ at fixed $r$) for all $r>0$ it holds, as $|\alpha|\to\infty$,
			\begin{equation}\label{eq:Fasymp}
			F_{n,\lambda}(r)= C_{n} \frac{\re[\cosh(r)^{i\alpha}]}{|\alpha|^\sigma \sinh(r)^\sigma}
			+o_{n,r}(|\alpha|^{-\sigma});
			\end{equation}
			\item (approximation with Bessel functions) as $|\alpha|\to\infty$, uniformly in $r> 0$,
			\begin{multline}\label{eq:Fasympclose}
			F_{n,\lambda}(r)=\frac{(2\pi)^{n/2}}{\omega_{n-1}} \sqrt{\frac{r}{\sinh r}}\\
			\cdot \pa{\alpha-\frac1{2i}}^{1-n/2}(\sinh r)^{1-n/2}
			J_{n/2-1}(\alpha r)\pa{1+O_n(1/|\alpha|)}.
			\end{multline}	
		\end{enumerate}
	\end{lemma}  
	
	\begin{remark} The \emph{Harish-Chandra function} $c_n(\alpha)$ (\emph{cf.} \cite[I.4]{Helgason2000})
		is closely related to the spectral density (introduced in \autoref{thm:spectrum}) by
		\begin{equation*}
		\rho_n(\alpha)=\frac{2^{n-2}}{\omega_{n-1}^2 |c_n(\alpha)|^2}.
		\end{equation*}
		As a function of $\alpha$, $c_n(\alpha)$ can be extended to a meromorphic function on $\C$ with poles on $i\N$.
	\end{remark}
	
	In what follows we regard $F_{\lambda,n}(r)$ as a hypergeometric function
	and make use of a number of known facts on that kind of special functions.
	Section \ref{sec:specialfunctions} collects the formulae we are using, together with precise
	bibliographic references.
	
	We begin by recalling that $F_{\lambda,n}(r)$ is the unique (smooth) solution of the ODE 
	\begin{equation}\label{eq:FODE}
	F_{n,\lambda}(r)''+2\sigma\coth(r)F_{n,\lambda}(r)'+\lambda F_{n,\lambda}(r)=0,
	\quad F_{n,\lambda}(0)=1,\, F_{n,\lambda}'(0)=0,
	\end{equation}
	on the positive real axis $r>0$.
	This is the hyperbolic analogue of the ODE characterizing radial solutions of Helmholtz equation on $\R^n$,
	the Euclidean case being recovered by replacing $\coth(r)$ with $r$.
	
	\begin{remark}
		A relevant difference with Euclidean setting:
		the function $z^{-\nu}J_\nu(z)$ appearing in the covariance of Berry's model \eqref{eq:berrycov} 
		is an entire function on $\C$ coinciding with its power series expansion at $0$, 
		whereas a power series of $F_{n,\lambda}$ at $0$ that one can derive from \eqref{eq:FODE}
		has a finite radius of convergence of order $O(\lambda^{-2\sigma})$ as $\lambda\to\infty$
		(\emph{cf.} \cite[15.2]{Olver2010}).
	\end{remark}
	
	When rewritten in terms of the variable $\sinh^2(r)$, \eqref{eq:FODE} 
	becomes a special case of the hypergeometric equation \eqref{eq:hypergeomODE}:
	it is thus possible to represent $F_{n,\lambda}$ with the principal branch of the hypergeometric function
	\eqref{eq:hypergeomDEF},
	\begin{equation}\label{eq:Fhypergeom}
	F_{n,\lambda}(r)={}_2F_1\pa{\frac{\sigma+i\alpha}2,\frac{\sigma-i\alpha}2,\frac{n}2,-\sinh^2(r)}
	\end{equation}
	(the branch cut of ${}_2F_1$, with respect to its last variable, is $[1,\infty]$).
	We refer to \cite[Section 4.1-4.2]{Borthwick2016} for more details on the representation with hypergeometric functions
	and asymptotics at the boundary on the half-plane model for $n=2$.
	
	\begin{remark}
		The arguments in the remainder of the paper rely on asymptotic properties of $F_{n,\lambda}(r)$,
		which we often deduce comparing \eqref{eq:Fhypergeom} with other special functions, 
		namely Legendre and Bessel functions.
		It is of course possible to obtain those asymptotics directly from the definition and basic properties
		of hypergeometric functions: we refer to \cite{Jones2001} for a proper discussion.
		We choose to proceed through comparison with other special functions appearing in
		the study of random waves on different geometries in order to emphasize analogies
		for readers who are familiar with those other models.
	\end{remark}
	
	\begin{proof}[Proof of \autoref{lem:Fasymp}, item (2)]
		Rewrite the representation \eqref{eq:Fdef2} as an oscillatory integral:
		denoting $t(r)=\tanh(r)$ to lighten notation,
		\begin{gather*}
		F_{n,\lambda}(r)=\frac{\omega_{n-2}}{2\omega_{n-1}} \cosh(r)^{i\alpha-\sigma}
		\int_{-\pi}^\pi e^{i\beta \phi(r,\theta)}\psi_n(r,\theta)d\theta,\\
		\beta=\alpha t(r),\quad
		\phi(r,\theta)=\frac{\log(1+t(r)\cos\theta)}{t(r)},\quad 
		\psi_n(r,\theta)=\frac{\sin(\theta)^{n-2}}{\pa{1+t(r)\cos\theta}^\sigma}.
		\end{gather*}
		Since $|t(r)|<1$ for all $r\in\R$ and $t(r)\sim r$ as $r\to 0$, the phase $\phi$ is uniformly bounded in both variables.
		As a function of $\theta$, it has two simple stationary points at $\theta=0$ and $\pi$, in which
		\begin{equation*}
		\left.\frac{d}{d\theta}\phi(r,\theta)\right|_{\theta=0,\pi}=0,\quad  
		\left.\frac{d^2}{d^2\theta}\phi(r,\theta)\right|_{\theta=0,\pi}=\frac{t(r)\pm 1}{(1\pm t(r))^2}\neq 0, \quad r\in\R.
		\end{equation*}
		For all $r\in\R$, the amplitude $\psi_n(r,\theta)$ is a smooth function of $\theta$, vanishing at
		critical points of the phase, $\theta=0,\pi$, with order $n-2$ and leading coefficient $(1\pm t(r))^{-\sigma}$.
		By the standard stationary phase method, see \cite[Section 6.1]{Bleistein1986}, we deduce the following asymptotic
		\begin{equation*}
		\int_{-\pi}^\pi e^{i\beta \phi(r,\theta)}\psi(r,\theta)d\theta=C_n \beta^{-\sigma}+ o_n(\beta^{-\sigma}),
		\quad \beta\to\infty
		\end{equation*}
		(here $C_n\in\C\setminus\set{0}$, depending on $n$ only).
	\end{proof}
	
	\begin{proof}[Proof of \autoref{lem:Fasymp}, item (1)]
		A power series expansion at large $r$ for ${}_2F_1$ can be deduced from \eqref{eq:hypergeomDEF}
		and \eqref{eq:hypergeomlin1},
		yielding an asymptotic for $r\to\infty$ at any fixed $\alpha\in\R\setminus\set{0}$,
		\begin{equation}\label{eq:harishchandra}
		F_{n,\lambda}(r)=\re\pa{c_n(\alpha)(\sinh r)^{i\alpha-\sigma}} +O_\lambda(\sinh(r)^{1-\sigma}),
		\end{equation}
		where $c_n(\alpha)$ is the \emph{Harish-Chandra function} defined in \eqref{eq:hcfunc}.
		Since $\sinh(r)\simeq e^r$ as $r\to\infty$,
		we can rewrite the asymptotic expression \eqref{eq:harishchandra} as
		\begin{equation*}
		F_\lambda(r)=\sinh(r)^{-\sigma}\pa{c_1(\alpha)\cos(\alpha r)+c_2(\alpha)\sin(\alpha r)}+O_\lambda(e^{-(\sigma+1)r}), 
		\end{equation*}
		with $c_1$ and $c_2$ real functions of $\alpha$. 
		By standard properties of Gamma function (see \cite[5.3]{Olver2010}) 
		we have that $|c_n(\alpha)|$ is uniformly bounded\footnote{
			We refer to \cite[Eqs. 4.4',4.4'']{Strichartz1989} for representations of $|c_n(\alpha)|$,
			for instance in terms of infinite products.} 
		in $\alpha$ away from $0$, so functions $c_1,c_2$ are uniformly bounded away from $0$. 
		
		Since $F_{n,\lambda}$ is a smooth solution of the ODE \eqref{eq:FODE}
		with $F_{n,\lambda}(0)=1$, it is bounded in a neighborhood of $0$:
		combined with the last displayed equation this proves \eqref{eq:Funif}.
		We will now bound the deviation from the main asymptotic by means of the ODE.
		Equation \eqref{eq:FODE} is put into canonical form by the substitution
		\begin{equation*}
		H(r)=\sinh(r)^\sigma F_\lambda(r), \quad H''(r)+b(r)H(r)=0, \quad b(r)=\alpha^2 -\frac{\sigma(\sigma-1)}{\sinh(r)^2}.
		\end{equation*}
		This is the reason why we expressed the exponential behavior of 
		$F_\lambda$ in $r$ in terms of $\sinh(r)$ above.
		Since $b(r)$ converges to $\alpha^2$ for large $r$, it is easy to compare $H$ with the harmonic oscillator of frequency $\alpha$.
		We thus define the remainder
		\begin{equation}\label{eq:remainder}
		R_\lambda(r)=\sinh(r)^\sigma F_\lambda(r)-\pa{c_1(\alpha)\cos(\alpha r)+c_2(\alpha)\sin(\alpha r)},
		\end{equation}
		which we know to be $O_\lambda(e^{-r})$ by the asymptotics above, and that satisfies
		\begin{equation*}
		R_\lambda''(r)+\alpha^2 R_\lambda(r)=\sigma(\sigma-1)\sinh(r)^{\sigma-2}F_\lambda(r),
		\end{equation*}
		as an immediate consequence of the ODE for $H$. The general solution for $R_\lambda(r)$ is given by
		\begin{equation*}
		R_\lambda(r)=A_1 \cos(\alpha r)+A_2 \sin(\alpha r) -\frac1\alpha\int_r^\infty \frac{\sin(\alpha(t-r))}{\sinh(t)^{2-\sigma}}F_\lambda(t)dt,
		\end{equation*}
		in which we can immediately tell that $A_1=A_2=0$ because of the known asymptotics of $F_\lambda$ and $R$ as $r\to\infty$.
		Substituting the definition \eqref{eq:remainder} of $R_\lambda$, we obtain the integral equation
		\begin{multline*}
		R_\lambda(r)=\frac1\alpha \int_r^\infty 
		\frac{\sin(\alpha(t-r))\pa{c_1(\alpha)\cos(\alpha t)+c_2(\alpha)\sin(\alpha t)}}{\sinh(t)^2}dt\\
		-\frac1\alpha \int_r^\infty 
		\frac{\sin(\alpha(t-r))R_\lambda(t)}{\sinh(t)^2}dt.
		\end{multline*}
		The first summand on the right-hand side can be controlled by means of van der Corput Lemma
		(see \cite[Chap. VIII, Sec. 1.2]{Stein1993}, we omit the elementary computation), leading to the estimate
		\begin{equation*}
		|R_\lambda(r)|\lesssim \frac{1}{\alpha^2\sinh(r)^2}+\frac1\alpha \int_r^\infty 
		\frac{|R_\lambda(r)|}{\sinh(t)^2}dt,
		\end{equation*}
		to which we apply Gr\"onwall inequality (we omit again an elementary computation) concluding
		\begin{equation*}
		|R_\lambda(r)|\lesssim \frac{1}{\alpha^2\sinh(r)^2}.
		\end{equation*}
		
		We thus obtained, for $r\geq 0$,
		\begin{equation}\label{eq:errorcontrol}
		\abs{\sinh(r)^\sigma F_\lambda(r)-\pa{c_1(\alpha)\cos(\alpha r)+c_2(\alpha)\sin(\alpha r)}}
		\lesssim \frac{1}{\alpha^2\sinh(r)^2}.\qedhere
		\end{equation}
	\end{proof}
	
	\begin{proof}[Proof of \autoref{lem:Fasymp}, item (3)]
		For any $n\geq 2$, the integral expression \eqref{eq:Fdef2} of $F_{n,\lambda}$
		coincides with a Legendre function of complex degree: from \eqref{eq:legendreINT} we derive
		\begin{equation}\label{eq:Flegendre}
		F_{n,\lambda}(r)=\pa{\frac{2}{\sinh r}}^{n/2-1} \Gamma\pa{\frac{n}{2}} P^{1-n/2}_{i\alpha-1/2}(\cosh r).
		\end{equation}
		Expressing the latter in terms of hypergeometric functions can be done combining \eqref{eq:legendreDEF}
		and \eqref{eq:legendreINT}, thus deriving the hypergeometric representation \eqref{eq:Fhypergeom} in a different way.
		From \eqref{eq:Flegendre}, the thesis follows by a straightforward application of \eqref{eq:legendreandbessel}.
	\end{proof}
	
	Item (3) of \autoref{lem:Fasymp} conveniently relates the hyperbolic spherical function
	with Bessel's function appearing in \eqref{eq:berrycov}, in the high-frequency limit.
	In fact, \eqref{eq:Fasympclose} plays a role akin to the one of Hilb's asymptotic for Legendre polynomials
	in the context of random wave models on the sphere $S^n$, see \cite{Marinucci2015},
	and we can immediately deduce the result on local behavior of $u_\lambda$ we stated in \autoref{prop:localbehavior}.

	\begin{proof}[Proof of \autoref{prop:localbehavior}]
		Denote
		\begin{equation*}
		d_{r,\lambda}= d\pa{\exp_x(v/\sqrt \lambda),\exp_x(v'/\sqrt \lambda)}.
		\end{equation*}
		By invariance under rotations we reduce ourselves to the case
		in which $v,v'\in\R^n$ lie on the plane $\R^2\times \set{0}\subset\R^n$ and 
		have polar coordinates respectively $(r,0),(cr,\theta)$, 
		with $r>0$, $0<c<1$ and $\theta\in [0,\pi]$. 
		The exponential map becomes identity when expressed in polar coordinates
		both on its domain and on the image $\H^n$ (\emph{i.e.} with hyperbolic polar coordinates centered at $x$).
		By the first hyperbolic law of cosines
		\begin{align*}
		\cosh {d_{r,\lambda}}
		&=\cosh\frac{r}{\sqrt\lambda}\cosh\frac{cr}{\sqrt\lambda}
		-\sinh\frac{r}{\sqrt\lambda}\sinh\frac{cr}{\sqrt\lambda}\cos\theta\\
		&=1+\frac{r^2}{2\lambda}\pa{1+c^2-2c\cos\theta}+O\pa{\frac{r^4}{\lambda^2}},
		\end{align*}
		the second step coming from first order expansion in $r/\sqrt\lambda=o(1)$. 
		Notice that $1+c^2-2c\cos\theta=|1-ce^{i\theta}|^2$ is uniformly bounded by 2.
		We are thus expressing the fact that the local behavior of hyperbolic and Euclidean distance is the same:
		\begin{equation*}
		d_{r,\lambda}
		=\frac{r}{\sqrt\lambda}|1-ce^{i\theta}|+O\pa{\frac{r^3}{\lambda^{3/2}}}
		=\frac{|v-v'|}{\sqrt\lambda}+O\pa{\frac{r^3}{\lambda^{3/2}}}, \quad \lambda\to\infty.
		\end{equation*}
		
		We conclude the proof combining \eqref{eq:Fasympclose} with the asymptotic we obtained for $d_{r,\lambda}$:
		\begin{multline*}
		F_\lambda(d_{r,\lambda})= \frac{(2\pi)^{n/2}}{\omega_{n-1}}
		\pa{\frac{\alpha|v-v'|}{\sqrt\lambda}}^{1-n/2}
		J_{n/2-1}\pa{\frac{\alpha|v-v'|}{\sqrt\lambda}}\pa{1+O(r/|\alpha|)}\\
		=C_{n,1}(v,v') \pa{1+O(r/|\alpha|)},
		\end{multline*}
		where the last step simply uses the definition of Berry's covariance function \eqref{eq:berrycov}
		and $\sqrt{\lambda}=\sqrt{\alpha^2+\sigma^2}\simeq \alpha$.
	\end{proof}
	
	\subsection{Double Integrals on Hyperbolic Balls}\label{ssec:changeofvariables}
	Given a measurable function $f:[0,\infty)\to [0,\infty)$, we are interested in the integral of $f(d(z,w))$
	over $x,y\in B_R$, the ball of $\H^n$ of radius $R$
	(to fix ideas one can consider the one centered at the origin)
	with respect to the hyperbolic volume, that is
	\begin{equation*}
	I(f,R)=\int_{B_R}\int_{B_R} f(d(x,y)) dm_n(x)dm_n(y).
	\end{equation*}
	To deal with this kind of integrals, we use once again hyperbolic polar coordinates.
	Let $(r,\vartheta)$ be polar coordinates of $x$ with respect to the center of $B_R$,
	and $(s,\vartheta')\in [0,\infty)\times S^{n-1}$ be coordinates of $y$ with respect to $x$;
	we can write
	\begin{equation*}
	I(f,R)
	= C_n \int_{S^{n-1}}d\varsigma_{n-1}(\vartheta)\int_0^R \sinh(r)^{n-1}dr \int_{B_R}  dm_n(y) f(d(x,y)),
	\end{equation*}
	in which the innermost integral does not depend on $\vartheta\in S^{n-1}$,
	thus neither does the integral in $dr$.
	We can thus fix a particular (arbitrary) $\vartheta$, 
	replace the outer integral with $\omega_{n-1}$ and write
	\begin{equation*}
	I(f,R)=C_n \int_0^R \sinh(r)^{n-1}dr \int_{S^{n-1}}d\varsigma_{n-1}(\vartheta') \int_0^{R(r,\vartheta')} \sinh(s)^{n-1}ds f(s),
	\end{equation*}
	where $R(r,\vartheta')$ is the hyperbolic length of 
	the geodesic arc leaving $x=(r,\vartheta)$ in direction $\vartheta'$
	and ending and the boundary $\partial B_R$.
	The function $R(r,\vartheta')$ of course varies between its minimum $R(r,\vartheta)=R-r$ and maximum $R(r,-\vartheta)=R+r$.
	If $f$ is non-negative, replacing $R(r,\vartheta')$ with one of these two values
	(and integrating out the now muted variable $\vartheta'$) 
	provides respectively a lower and an upper bound for the integral.
	
	\begin{remark}
		By means of hyperbolic trigonometry (we refer to \cite[Section 5.6]{Anderson2005}) 
		one can explicitly represent the function $R(r,\vartheta')$;
		for instance in the case $n=2$, fixing $\vartheta=0\in [0,2\pi]\simeq S^1$, it holds
		\begin{equation*}
		\cosh\pa{R(r,\vartheta')}=\frac{\cosh(R)\cosh(r)+\cos(\vartheta') \sinh(r)\sqrt{\sinh^2(R)-\sin^2(\vartheta')\sinh^2(r)}}{1+\sin^2(\vartheta') \sinh^2(r)},
		\end{equation*}
		$r>0, \vartheta'\in [0,2\pi]\simeq S^1$.
	\end{remark}
	
	\subsection{Asymptotics for large \texorpdfstring{$\lambda$}{lambda}, fixed \texorpdfstring{$R$}{R}}
	When the domain of integration is fixed, for moments of order $q\geq 2$
	we can determine the asymptotic of $C^{n,q}_{R,\lambda}$ as $\lambda\uparrow\infty$
	by separating the contributions of regions in which integrating variables $x,y\in B_R$ are respectively closer 
	or farther than $1/\alpha$ with respect to each other.
	To this end, the above change of variables reduces our task to control one-dimensional integrals
	close to and far from the lower extremum of integration.
	We can proceed this way since for $q\geq 2$ we can derive good controls simply from the exponential decay of
	$F_{n,\lambda}$. 
	
	For odd $q\geq 3$, oscillations provide additional cancellations
	and we will thus establish only an upper bound that suffices to our needs.
	However, in the special case of $q=1$ the decay of $F_{n,\lambda}$ is outweighed by
	the volume element, so the oscillations play a determinant role and can not be neglected.
	We thus need a more careful analysis in that case, 
	so we will discuss it separately by means of Fourier analysis on $\H^n$.
	
	\begin{lemma}\label{lem:momentlambda}
		As $\lambda\to\infty$, for a fixed $R>0$, we have the following asymptotics:
		\begin{itemize}
			\item ($q=1$, any $n\geq 2$) $C_{R,\lambda}^{n,1}\lesssim_R |\alpha|^{-2\sigma-2}$;
			\item ($q=2$, any $n\geq 2$) $C_{R,\lambda}^{n,2}\simeq_R |\alpha|^{-2\sigma}$;
			\item (even $q\geq 2$, any $n\geq 2$) $C_{R,\lambda}^{n,q}\simeq_R |\alpha|^{-2\sigma-1}$
			\emph{except for the single case} ($n=2,q=4$) for which $C_{R,\lambda}^{2,4}\simeq_R |\alpha|^{-2}\log(|\alpha|)$;
			\item (odd $q\geq 2$, any $n\geq 2$) $C_{R,\lambda}^{n,q}\lesssim_R |\alpha|^{-2\sigma-1}$
			\emph{except for the single case} ($q=n=3$) for which $C_{R,\lambda}^{3,3}\lesssim_R |\alpha|^{-3}\log(|\alpha|)$.
		\end{itemize}
	\end{lemma}
	
	\begin{remark}\label{rmk:n=q=3}
		The two exceptional cases ($n=2,q=4$) and ($q=n=3$) correspond, 
		as we show below, to critical choices of the parameters.
		However, for ($n=2,q=4$) a logarithmic correction appears in the \emph{true} asymptotic, 
		whereas in the other case the factor $\log \alpha$ appearing in the upper bound we state
		is \emph{not} present in the real asymptotic, 
		as it is canceled by oscillations that persist due to $q=3$ being odd.
		This last fact is irrelevant in our scope, a careful (and lengthy) estimate would be required to prove it,
		so we do not discuss it any further.
	\end{remark}
	\noindent
	We will assume $\alpha>0$, the negative case being identical.
	
	\begin{proof}[Proof of case $q=1$.]
		Assume that $B_R$ is centered at the origin of $\H^n$;
		since the indicator function $\one_{B_R}(x)=\one_{[0,R]}(d(x,o))$ is radial, so is its Fourier transform
		\begin{align*}
		f_R(\alpha)&=\F\one_{B_R}(\alpha,\vartheta)=\int_{B_R} \e_n(x,-\alpha,\vartheta)dm_n(x)
		=c_n \int_0^R F_{n,\lambda}(r) (\sinh r)^{n-1} dr\\
		&=c_n \int_0^R \int_0^\pi \pa{\cosh r-\sinh r \cos\theta}^{-\sigma+i\alpha} (\sin \theta)^{n-2}(\sinh r)^{n-1} d\theta dr, 
		\end{align*}
		(the latter descending directly from definitions in Section \ref{sec:Hd} and change of variables).
		Standard stationary phase analysis (see \cite[8.4]{Bleistein1986})
		allows to obtain asymptotics of the oscillatory integral in display:
		we have $f_R(\alpha)=O_R(\alpha^{-\sigma-1})$ for $\alpha\to\infty$
		(see Subsection \ref{ssec:rasymptotic} for completely analogous computations with details).
		
		We recall \eqref{eq:Fdef1} in the form
		\begin{equation*}
		F_{n,\lambda}(d(x,y))
		=\frac1{\omega_{n-1}}\int_{S^{n-1}} \e_n(x,\alpha,\vartheta)\e_n(y,-\alpha,\vartheta) d\varsigma_{n-1}(\vartheta).
		\end{equation*}
		Finally, we will use the (formal) orthogonality relation
		\begin{equation*}
		\rho_n(\alpha)\int_{\H^n} \e_n(x,\alpha,\vartheta)\e_n(x,-\alpha',\vartheta')dm_n(x)
		=\delta_{\alpha=\alpha'}\delta_{\vartheta=\vartheta'}
		\end{equation*}
		in our computation: it can be derived by applying $\F \F^{-1}$ to the right-hand
		side, regarded as a Dirac delta on $\R^+\times S^{n-1}$. 
		A standard mollification argument 
		--or the extension of Fourier transform to distributions-- allows to make it rigorous,
		we omit details for the sake of brevity.
		Combining the formulas above and Fourier inversion (\autoref{prop:fourier}), the thesis follows from:
		\begin{align*}
		C_{R,\lambda}^{n,1}
		&= \int_{(\H^n)^2}\one_{B_R}(x)\one_{B_R}(y) F_{n,\lambda}(d(x,y)) dm_n(x)dm_n(y)\\
		&= C_n \int_{(\H^n)^2}dm_n^2(x,y) \int_{(S^{n-1})^3} d\varsigma_{n-1}^3(\vartheta,\vartheta',\vartheta'')
		\int_{(\R^+)^2}\rho_n(\alpha')\rho_n(\alpha'')d\alpha'd\alpha''\\
		&\qquad \times f_R(\alpha')f_R(\alpha'') \e_n(x,\alpha',\vartheta') \e_n(y,\alpha'',\vartheta'')
		\e_n(x,\alpha,\vartheta)\e_n(y,-\alpha,\vartheta)\\
		&= C_n \int_{(S^{n-1})^3} d\varsigma_{n-1}^3(\vartheta,\vartheta',\vartheta'')
		\int_{(\R^+)^2}\rho_n(\alpha')\rho_n(\alpha'')d\alpha'd\alpha''\\
		&\qquad \times f_R(\alpha')f_R(\alpha'') \delta_{\alpha=-\alpha'}\delta_{\vartheta=\vartheta'}
		\delta_{\alpha=\alpha''}\delta_{\vartheta=\vartheta''} \rho_n(\alpha)^{-2} = C_n |f_R(\alpha)|^2.
		\qedhere
		\end{align*}
	\end{proof}
	
	\begin{remark}\label{rmk:zerosvariance}
		The Fourier transform $f_R(\alpha)$ of $\one_{B_R}$ is in fact an oscillating function,
		and it has a discrete, countable set of zeros on $\R^+$.
		In sight of the discussion in Subsection \ref{ssec:polyspectra}, this means that the variance of $h^{n,1}_{R,\lambda}$
		vanishes for infinite values of $\alpha$, at which one can not divide by the variance in order to
		study a Central Limit Theorem.
		The same might happen for all other odd orders $q\geq 3$, and this is an open question
		even for random wave models on other geometries: it was conjectured in \cite{Marinucci2011} that variances of odd polyspectra do not vanish in the high 
		of random spherical
		Substantiating this claim would require a much more careful control of $C_{R,\lambda}^{n,q}$ for odd orders $q$,
		which we are not yet able to produce.
	\end{remark}
	
	\begin{proof}[Proof of case $q\geq 2$, upper bound.]
		By Subsection \ref{ssec:changeofvariables} and \autoref{lem:Fasymp},
		\begin{align}\label{eq:momentcomp1}
		C_{R,\lambda}^{n,q}	&\lesssim_{n,R} \int_0^{2R} |F_{n,\lambda}(s)|^q \sinh(s)^{2\sigma} ds\\ \nonumber
		&\lesssim_n \int_0^{1/\alpha} \sinh(s)^{2\sigma} ds
		+\frac1{\alpha^{q\sigma}} \int_{1/\alpha}^{2R} \sinh(s)^{\sigma (2-q)}ds=:(A)+(B),
		\end{align}
		where the second step uses the fact that $F_{n,\lambda}$ is uniformly bounded 
		to control the integral over small $s<1/\alpha$,
		and \eqref{eq:Fasymp} for large values of $s$, neglecting trigonometric oscillations.
		The asymptotic behavior as $\alpha\to \infty$ 
		is then determined by recalling that $\sinh(x)\sim x$ as $x\to 0$:
		summand $(A)$ on the right-hand side is of order $\alpha^{-2\sigma-1}$, 
		so we need to discuss whether $(B)$ provides a relevant correction.
		
		The integrand of $(B)$ is integrable at $0$ when $\sigma (2-q)>-1$,
		which is the case only when:
		\begin{itemize}
			\item ($q=2$, any $n\geq 2$) in this case it actually is $\sigma (2-q)=0$, the second integral
			can be bounded with the one over $[0,2R]$, thus it is of order $|\alpha|^{-2\sigma}$,
			and it prevails in the asymptotic;
			\item ($q=3,n=2$) that is $\sigma (2-q)=-1/2<0$,
			so summand $(B)$ does not provide a correction to $(A)$.
		\end{itemize}
		The case $\sigma (2-q)=-1$ is realized in two cases:
		\begin{itemize}
			\item ($q=4,n=2$)
			\begin{equation*}
			C_{R,\lambda}^{2,4}\lesssim_R \frac1{\alpha^2} +
			\frac1{\alpha^2} \int_{1/\alpha}^{2R} \sinh(s)^{-1}ds\simeq_R \frac{\log \alpha}{\alpha^2};
			\end{equation*}
			\item ($q=n=3$)
			\begin{equation*}
			C_{R,\lambda}^{3,3}\lesssim_R \frac1{\alpha^3} +
			+\frac1{\alpha^3} \int_{1/\alpha}^{2R} \sinh(s)^{-1}ds\simeq_R \frac{\log \alpha}{\alpha^3} \lesssim_R \frac1{\alpha^3}.
			\end{equation*}
		\end{itemize}
		Finally, if $\sigma (2-q)<-1$,
		\begin{equation}\label{eq:momentcomp2}
		(B)=\frac1{\alpha^{q\sigma}} \int_{1/\alpha}^{2R} \sinh(s)^{\sigma (2-q)}ds\simeq_R \frac{1}{\alpha^{2\sigma+1}}
		\end{equation}
		has the same asymptotic behavior of $(A)$.
	\end{proof}
	
	\begin{proof}[Proof of case $q\geq 2$ and even, lower bound.]
		Once again we apply the argument of Subsection \ref{ssec:changeofvariables}:
		this time we replace $R(r,\theta)$ with its minimum $R-r$, 
		\begin{multline}\label{eq:lowerboundC}
		C_{R,\lambda}^{n,q}
		\gtrsim \int_0^R \int_0^{R-r} F_{n,\lambda}(s)^q \sinh(s)^{2\sigma} \sinh(r)^{2\sigma} ds dr\\
		=\int_0^R g_R(s) \sinh(s)^{2\sigma} F_{n,\lambda}(s)^q ds,\\
		g_{n,R}(s)=\int_0^{R-s} \sinh(r)^{2\sigma}dr.
		\end{multline}
		Once again we divide the integral in $ds$ over intervals $[0,1/\alpha]$ and $[1/\alpha,R]$.
		In sight of the upper bound, in the cases $q=2$, all $n\geq 2$, 
		and $q=4,n=2$ we can actually neglect the contribution of $[0,1/\alpha]$:
		applying \eqref{eq:Fremainder} for $s>1/\alpha$ we obtain
		\begin{align*}
		(q=2)&\quad C_{R,\lambda}^{n,2}\gtrsim_R \frac1{\alpha^{2\sigma}} 
		\int_{1/\alpha}^{R} g_{n,R}(s) \re\bra{\cosh(s)^{i\alpha}}^2  ds\simeq_R \frac1{\alpha^{2\sigma}},\\
		(q=4, n=2)&\quad C_{R,\lambda}^{2,4}\gtrsim_R \frac1{\alpha^{2}} \int_{1/\alpha}^{R} \frac{g_{n,R}(s)}{\sinh(s)}
		\re\bra{\cosh(s)^{i\alpha}}^4 ds \simeq_R \frac{\log\alpha}{\alpha^2}.
		\end{align*}
		Derivation of asymptotics on the right-hand side is as follows:
		since $g_{n,R}(s)$ vanishes at $s=R$,
		once again the relevant contribution comes from the lower integration extremum,
		close to which $g_{n,R}$ is bounded from below (in terms of $n,R$),
		while $\cosh(s)^{i\alpha}\sim 1$ for small $s$.
		We omit the tedious, but elementary computation.

		When $q\geq 6$, the contribution relative to $s\in [0,1/\alpha]$ is the relevant one, 
		matching the asymptotic obtained in the upper bound. 
		Indeed, for all $q\geq 2$ and $\epsilon>0$, $|F_{n,\lambda}(s)|^q\geq 1-\epsilon$ on $s\in [0,1/\alpha]$
		if $\alpha$ is large enough, since $F_{n,\lambda}(0)=1$ and $F_{n,\lambda}$ is continuous, so we can bound	
		\begin{align*}
		C_{R,\lambda}^{n,q}
		\gtrsim \int_0^R g_R(s) \sinh(s)^{2\sigma} F_{n,\lambda}(s)^q ds 
		\gtrsim_R (1-\epsilon) \int_0^{1/\alpha} \sinh(s)^{2\sigma} ds\simeq \frac1{\alpha^{2\sigma+1}}.
		\end{align*}
		Therefore, there is actually no need to estimate a lower bound on the contribution of $s\in [1/\alpha,R]$.
	\end{proof}
	
	\begin{lemma}\label{lem:taillambda}
		Let $n\geq 2$, $R>0$, $q\in 2\N$ and
		\begin{equation*}
		g_q(x)=\sum_{\ell=q/2}^\infty \binom{2\ell}{\ell}\frac{x^{2\ell}}{4^\ell},
		\quad x\in (-1,1).
		\end{equation*}
		As $\lambda\to\infty$ it holds
		\begin{equation}\label{eq:analyticlambda}
		\int_{B_R^2} g_q(F_{n,\lambda}(d(x,y))) dm_n^{\otimes 2}(x,y)
		\lesssim_{n,R} C^{n,q}_{R,\lambda}.
		\end{equation}
	\end{lemma}
	
	\begin{remark}\label{rmk:Fnever1}
		It is worth observing that $F_{n,\lambda}(r)=0$ if and only if $r=0$, for any $n,\lambda$;
		we report here a direct argument to prove it.
		Take points $x\in \H^n$ and $y=(1,0,\dots,0)\in\H^n$ with $x_0=r$, so that $F_{n\lambda}(d(x,y))=F_{n\lambda}(r)$,
		and assume the latter to equal 1. Then, by \eqref{eq:Fdef1}, we have
		\begin{equation*}
		1=F_{n\lambda}(r)=\frac{1}{\omega_{n-1}}\int_{S^{n-1}}[x,(1,\vartheta)]^{-\sigma+i\alpha}d\varsigma_{n-1}(\vartheta),
		\end{equation*} 
		and since $[x,y]\geq 1$ for any $x,y\in \H^n$, thus $|[x,(1,\vartheta)]^{-\sigma+i\alpha}|\leq 1$,
		we deduce that actually it must hold $|[x,(1,\vartheta)]^{-\sigma+i\alpha}|=1$, thus $|[x,(1,\vartheta)]|= 1$,
		for almost all $\vartheta\in S^{n-1}$, which implies that $r=x_0=0$.
	\end{remark}
	
	\begin{proof}
		The power series defining $g_q$ has radius of convergence 1, 
		hence the integrand in \eqref{eq:analyticlambda} is well-defined for all distinct
		$x,y$, that is $m_n^{\otimes 2}$-almost everywhere, thanks to \autoref{rmk:Fnever1}.
		Denote by $S=S(n,R,\lambda,q)$ the left-hand side of \eqref{eq:analyticlambda}.
		As in \eqref{eq:momentcomp1} we split the integration domain,
		\begin{align*}
		S &\lesssim_{n,R} \int_0^{2R} g_q(F_{n,\lambda}(s))\sinh(s)^{2\sigma} ds  \\ \nonumber
		&\lesssim_{n,R} \int_0^{1/\alpha} \frac{\sinh(s)^{2\sigma} ds}{\sqrt{1-F_{n,\lambda}(s)^2}}
		+\int_{1/\alpha}^{2R} F_{n,\lambda}(s)^q \sinh(s)^{2\sigma} ds:=(A)+(B),
		\end{align*}
		in which:
		\begin{itemize}
			\item we applied the inequality $|g_q(x)|\leq 1/\sqrt{1-x^2}$ on the interval $[0,1/\alpha]$
			($g_q$ is the tail of the series \eqref{eq:powerseries} having only positive terms);
			\item we replaced $g_q(F_{n,\lambda})$ with $F_{n,\lambda}^q$ in the integral over $[1/\alpha,R]$
			thanks to the fact that $|F_{n,\lambda}(s)|\leq C_{n,R}<1$ uniformly on $s\in [1/\alpha,R]$,
			as it directly descends from either \eqref{eq:Fremainder} or \eqref{eq:Fasymp},
			and observing that $x^{-q}g_q(x)$ is uniformly bounded on compacts of $[0,1)$.
		\end{itemize}
		It holds $(B)\simeq_{n,R} C^{n,q}_{R,\lambda}$ (see the proof of \autoref{lem:momentlambda}),
		so we are left to control $(A)$.
		
		We apply the asymptotic \eqref{eq:Fasympclose} to control $F_{n,\lambda}$ close to $s=0$.
		From \eqref{eq:Fasympclose} and the definition of $\sinh(r)$ we derive
		\begin{multline*}
		F_{n,\lambda}(s)=\frac{(2\pi)^{n/2}}{\omega_{n-1}}(\alpha s)^{1-n/2}J_{n/2-1}(\alpha s)(1+O(1/\alpha))\\
		=:f(\alpha s)(1+O(1/\alpha)), 
		\quad s\in [0,1/\alpha],
		\end{multline*}
		where we stress that Landau's $O$ does not depend on parameters.
		From the definition of Bessel functions in \eqref{eq:besselDEF} we know that 
		$f(\alpha s)=1-\alpha^2 s^2+\dots$ is (a power series in $\alpha s$) analytic on the whole $\R$.
		This, once again recalling that $\sinh(r)\simeq r$, $r\to 0$, 
		implies that the integral $(A)$ converges and that
		\begin{equation*}
		(A)\lesssim_{n,R} \int_0^{1/\alpha} \frac{\sinh(s)^{2\sigma} ds}{\alpha s}
		\simeq_{n,R} \frac{1}{\alpha^{2\sigma+1}}=O(C^{n,q}_{R,\lambda}) 
		\end{equation*}
		for all $q$, the last step following from \autoref{lem:momentlambda}.
	\end{proof}
	
	Recalling the power series expansion \eqref{eq:powerseries}, \autoref{lem:taillambda} implies:
	
	\begin{corollary}\label{cor:sqrtlambda}
		Let $n\geq 2$, $R>0$;
		as $\lambda\to\infty$ it holds
		\begin{equation}\label{eq:sqrtlambda}
		\abs{\int_{B_R^2}
			\frac{dm_n^{\otimes 2}(x,y)}{\sqrt{1-F_{n,\lambda}(d(x,y))^2}}
			-m_n(B_R)^2- \frac12 \var h^{n,2}_{R,\lambda}}
		\lesssim_{n,R} \var h^{n,4}_{R,\lambda},
		\end{equation}
		in particular, the first summand of the left-hand side is uniformly bounded in $\lambda$.
	\end{corollary}
	
	Moving to contractions in \autoref{eq:contractionsK}, we have the following:
	
	\begin{lemma}\label{lem:contractionslambda}
		Let $n\geq 2$, $R>0$, $p,p'\geq 1$ and even $2\leq q\leq p+p'$.
		As $\lambda\to\infty$ it holds
		\begin{multline}\label{eq:contractionbiglambda}
		\int_{B_R^4}
		F_{n,\lambda}(d(x,y))^p F_{n,\lambda}(d(y,z))^{p'}\\ 
		\cdot F_{n,\lambda}(d(z,w))^p F_{n,\lambda}(d(w,x))^{p'}
		dm_n^{\otimes 4}(x,y,z,w)
		=o_R\pa{\var (h^{n,q}_{R,\lambda})^2}.
		\end{multline}
	\end{lemma}
	
	\begin{proof}
		Denote by $I(\lambda)=I_{n,R,p,p'}(\lambda)$ the left-hand side of \eqref{eq:contractionbiglambda},
		assume without loss of generality $p\leq p'$
		and drop dependences of constants on $n,R,p,p'$ to lighten notation.
		Applying GM-QM inequality to the last two factors of the integrand we obtain
		\begin{equation*}
		I(\lambda) \lesssim \int_{B_R^4}
		|F_{n,\lambda}(d(x,y))|^p |F_{n,\lambda}(d(y,z))|^{p'}
		|F_{n,\lambda}(d(z,w))|^{2p} dm_n^{\otimes 4}(x,y,z,w)+\dots,
		\end{equation*}
		where dots stand for an analogous term with higher exponents (and thus lower asymptotic order). 
		Integration variables can now be decoupled: moving to polar coordinates
		we can write
		\begin{equation*}
		I(\lambda)\lesssim J_p(\lambda)J_{p'}(\lambda)J_{2p}(\lambda),\quad
		J_p(\lambda)=\int_0^R |F_{n,\lambda}(r)|^p\sinh(r)^{2\sigma}dr,
		\end{equation*}
		where we can apply the asymptotic estimates obtained in the proof of the previous Lemma,
		\begin{equation*}
		J_p(\lambda)
		\begin{cases}
		\lesssim \lambda^{-\sigma-1} &p=1,\\
		\simeq \lambda^{-\sigma} &p=2,\\
		\lesssim \lambda^{-\sigma-1/2} &p\geq 3\text{ excluding }n=2,p=4\text{ and } n=p=3,\\
		\lesssim \log(\lambda)\lambda^{-\sigma-1/2} &n=2,p=4\text{ or } n=p=3.
		\end{cases}
		\end{equation*}
		The thesis now follows from a case-by-case check, which we summarize in the following table
		(the right-most column coming once again from \autoref{lem:momentlambda}).
		\begin{center}
			\begin{tabular}{c|c|c|c}
				$q,n$			&	possible $p,p'$	&	$J_pJ_{p'}J_{2p}\lesssim$	&	$\var(h^{n,q}_{R,\lambda})^2\simeq$\\
				\hline
				$q=2$, all $n$	&	1,1		&	$\lambda^{-3\sigma-2}$	&	$\lambda^{-2\sigma}$	\\
				\hline
				$q=4,n=2$		&	1,3		&	$\lambda^{-3}$	&	$\lambda^{-2}\log(\lambda)^2$	\\
				&	2,2		&	$\lambda^{-2}$	&		\\
				\hline
				$q=4,n=3$		&	1,3		&	$\lambda^{-9/2}\log(\lambda)^2$	&	$\lambda^{-3}$	\\
				&	2,2		&	$\lambda^{-7/2}$	&		\\
				\hline
				$q=4,n\geq 4$,		 &	1,3		&	$\lambda^{-3\sigma-3/2}$	&	$\lambda^{-2\sigma-1}$	\\
				thus $\sigma\geq 3/2$&	2,2		&	$\lambda^{-3\sigma-1/2}$	&		\\
				\hline
				$q\geq 6$, all $n$&	$1,\geq 5$	&	$\lambda^{-3\sigma-3/2}$	&	$\lambda^{-2\sigma-1}$	\\
				&	$2,\geq 4$			&	$\lambda^{-3\sigma-1}\log(\lambda)^2$	&		\\
				&	both $\geq 3$		&	$\lambda^{-3\sigma-3/2}\log(\lambda)^4$	&		\\
			\end{tabular}
		\end{center}
		The thesis follows since asymptotics in the third column vanish faster then the ones in the fourth one;
		notice that in the case $q\geq 6$ we included in the third column possible logarithmic corrections
		occurring in low dimension since $\var(h^{n,q}_{R,\lambda})^2$ in this case vanishes much slower
		for all $n$.
		On the other hand, one should be aware that in the case $q=4,n=2$ a careful account of the logarithmic
		correction in the asymptotic of the variance is crucial.
	\end{proof}
	
	\subsection{Asymptotics for large $R$, fixed \texorpdfstring{$\lambda$}{lambda}}
	The main difference between the analysis in high-frequency regimes outlined above, and the one on large domains we carry through in this paragraph, is that in the latter case the oscillations of $F_{n,\lambda}$ (more pronounced close to $r=0$) actually play no role, and the determinant factor is the exponential decay at $r\to\infty$.
	
	\begin{lemma}\label{lem:momentr}
		As $R\to\infty$, for a fixed $\lambda> \sigma^2$ (\emph{i.e.} for fixed $\alpha\in \R^\ast$), 
		we have the following asymptotics:
		\begin{itemize}
			\item ($q=1$, any $n\geq 2$) $C_{R,\lambda}^{n,1}\lesssim_{n,\lambda} e^{2\sigma R}$;
			\item ($q=2$, any $n\geq 2$) $C_{R,\lambda}^{n,2}\simeq_{n,\lambda} R e^{2\sigma R}$;
			\item (even $q\geq 2$, any $n\geq 2$) $C_{R,\lambda}^{n,q}\simeq_{n,\lambda} e^{2\sigma R}$;
			\item (odd $q\geq 2$, any $n\geq 2$) $C_{R,\lambda}^{n,q}\lesssim_{n,\lambda} e^{2\sigma R}$.
		\end{itemize}
	\end{lemma}
	
	\begin{proof}
		Starting from the case $q=1$, we established in the previous Section that 
		\begin{gather*}
		C_{R,\lambda}^{n,1}=C_n \pa{\int_0^R F_{n,\lambda}(r) (\sinh r)^{2\sigma} dr}^2,
		\end{gather*}
		so we simply discuss the limiting behavior of the right-hand side as $R\to\infty$
		instead of $\lambda\to\infty$. In the large $R$ case, we are not dealing with an oscillatory integral,
		and the estimate is actually easier: by \eqref{eq:Fremainder} it holds
		\begin{multline*}
		\int_0^R F_{n,\lambda}(r) \sinh (r)^{2\sigma} dr\\
		=C_{n,\lambda} \int_0^R \re[c_n(\alpha) \sinh(r)^{i\alpha}]\sinh (r)^{\sigma}dr
		+O_{n,\lambda} \pa{\int_1^R \sinh(r)^{\sigma-2} dr}\\
		\lesssim_{n,\lambda} \int_0^R e^{\sigma r} \cos(\alpha r+\phi_n) dr  
		\end{multline*}
		for some phase $\phi_n\in [0,2\pi]$, from which the statement for $q=1$ directly follows.
		
		Moving to $q\geq 2$, let us establish upper bounds by neglecting oscillations as above.
		From the discussion in Subsection \ref{ssec:changeofvariables} we deduce
		\begin{equation*}
		C_{R,\lambda}^{n,q}\lesssim_n  \int_0^R dr \sinh(r)^{2\sigma} \int_0^{R+r} ds \sinh(s)^{2\sigma} |F_{n,\lambda}(s)|^q ,
		\end{equation*}
		At this point we can just apply a rough consequence of \eqref{eq:Fremainder},
		$|F_{n,\lambda}(s)|\lesssim_{n,\lambda} e^{-\sigma s}$, to deduce
		\begin{equation*}
		C_{R,\lambda}^{n,q} \lesssim_{n,\lambda} \int_0^R dr e^{2\sigma r} \int_0^{R+r} ds e^{\sigma(2-q)s},
		\end{equation*}
		from which upper bounds coherent with the statement directly follow.
		Just as in the high-frequency case, we only claimed exact asymptotics only for even $q\geq 2$,
		since lower bounds for odd $q$ are once again made harder to derive by oscillations of $F_{n,\lambda}$.

		As for lower bounds for \emph{even} $q\geq 2$, starting from the change of variables of Subsection \ref{ssec:changeofvariables}, this time we replace $R(r,\theta)$ with its minimum $R-r$, 
		\begin{gather*}
		C_{R,\lambda}^{n,q}
		\gtrsim \int_0^R \int_0^{R-r} F_{n,\lambda}(s)^q \sinh(s)^{2\sigma} \sinh(r)^{2\sigma} ds dr
		=\int_0^R g_R(s) \sinh(s)^{2\sigma} F_{n,\lambda}(s)^q ds,\\
		g_{n,R}(s)=\int_0^{R-s} \sinh(r)^{2\sigma}dr\simeq_n m_n(B_{R-s})\simeq_n e^{2\sigma (R-s)}.
		\end{gather*}
		The case $q\geq 4$ is the easier one: combining \eqref{eq:Fremainder} and the last displayed formulae,
		\begin{equation*}
		C_{R,\lambda}^{n,q}\gtrsim_{n,\lambda} 
		e^{2\sigma R}\int_0^R \re[c_n(\alpha) \sinh(s)^{i\alpha}]^q \sinh(s)^{\sigma q}ds
		\simeq_{n,\lambda}  e^{2\sigma R},
		\end{equation*}
		since the integrand consists in a positive power of a trigonometric function
		(as opposed to the case $q=1$ above) and a rapidly decreasing function,
		thus the integral is overall $O(1)$ as $R\to\infty$.
		
		The last estimate clearly holds true also for $q=2$, but in that case it does not capture the
		correct asymptotic behavior, due to the rough control from below in replacing $R(r,\vartheta)\mapsto R-r$.
		In this case we proceed more carefully: in the notation of Subsection \ref{ssec:changeofvariables},
		an easy geometric argument reveals that there exists a solid angle\footnote{
			Recall that, given a point $x$ whose distance from the origin is $r$, 	
			$R(r,\vartheta)$ is the geodesic distance of the boundary of $B_R$ from $x$ in direction $\vartheta$,
			so the solid angle $\Sigma$ is centered around the geodesic arc joining $x$ and the origin.
			As $R$ increases, one can actually take larger and larger $\Sigma$.}
		$\Sigma\subset S^{n-1}$ such that $R(r,\theta)\geq R$ for all $r\in[0,R]$ and $\vartheta\in\Sigma$,
		so that we can control
		\begin{align*}
		&\int_{B_R}dm_n(x)\int_{B_R}dm_n(y)F_{n,\lambda}(d(x,y))^2\\
		&\qquad\simeq_n \int_0^R \sinh(r)^{2\sigma}dr\int_{S^{n-1}}d\varsigma_{n-1}(\vartheta)
		\int_0^{R(r,\vartheta)} F_{n,\lambda}(s)^2\sinh(s)^{2\sigma}ds\\
		&\qquad\geq \varsigma_{n-1}(\Sigma) \int_0^R \sinh(r)^{2\sigma}dr
		\int_0^{R} F_{n,\lambda}(s)^2\sinh(s)^{2\sigma}ds\\
		&\qquad \simeq_{n,\lambda} m_n(B_R)
		\int_0^R \re[c_n(\alpha) \sinh(s)^{i\alpha}]^2ds\simeq_{n,\lambda} R\cdot m_n(B_R).\qedhere
		\end{align*}
	\end{proof}
	
	\begin{lemma}\label{lem:sqrtR}
		Let $n\geq 2$, $\lambda>\sigma^2$;
		as $R\to\infty$ it holds
		\begin{equation}\label{eq:sqrtR}
		\int_{B_R^2}
		\frac{dm_n^{\otimes 2}(x,y)}{\sqrt{1-F_{n,\lambda}(d(x,y))^2}}
		-m_n(B_R)^2- \frac12 \var h^{n,2}_{R,\lambda}
		\lesssim_{n,\lambda} m_n(B_R).
		\end{equation}
	\end{lemma}
	
	\begin{proof}
		We only provide a sketch of the proof: we can argue in the case $R\to \infty$
		just like we did for obtaining \autoref{cor:sqrtlambda}, that is using
		the arguments of \autoref{lem:momentlambda} on single polyspectra applied to
		power series in \autoref{lem:taillambda}.
		The thesis can be rephrased as
		\begin{equation*}
		\int_{B_R^2} g(F_{n,\lambda}(d(x,y))) dm_n^{\otimes 2}(x,y)
		\lesssim_{n,\lambda} C^{n,4}_{R,\lambda}, 
		\quad g(x)=\sum_{\ell=2}^\infty \binom{2\ell}{\ell}\frac{x^{2\ell}}{4^\ell},
		\quad x\in (-1,1).
		\end{equation*}
		Close to $d(x,y)=0$ we can apply an expansion
		deduced from \eqref{eq:hypergeomDEF} and \eqref{eq:Fhypergeom},
		\begin{equation*}
		F_{n,\lambda}(r)=1-\frac{\lambda}{2n}\sinh(r)^2+O_{n,\lambda}(\sinh(r)^4)
		\end{equation*}
		(keep in mind that $\lambda$ here is fixed),
		from which convergence of the integral follows (thanks to \autoref{rmk:Fnever1}).
		We can then restrict the integration domain to $D_R=B_R^2\setminus\set{d(x,y)>1}$,
		since the latter provides the relevant contribution as $R\to \infty$.
		By \autoref{lem:Fasymp}, $0<C_{n,\lambda}\leq F_{n,\lambda}(d(x,y))\leq C_{n,\lambda}'<1$ 
		for all $(x,y)\in D_R$, and moreover $F_{n,\lambda}(r)\lesssim_{n,\lambda}e^{-\sigma r}$
		as $r\to\infty$; this, together with the fact that $x^{-4}g(x)$ is bounded on compacts of $[0,1)$, 
		leads to the desired estimate with a direct computation.
	\end{proof}
	
	\begin{lemma}\label{lem:contractionsr}
		Let $n\geq 2$, $R>0$, $p,q\in\N_0$. As $R\to\infty$ it holds
		\begin{multline}\label{eq:contractionbigR}
		\int_{B_R^4}
		F_{n,\lambda}(d(x,y))^p F_{n,\lambda}(d(y,z))^q\\ 
		\cdot F_{n,\lambda}(d(z,w))^p F_{n,\lambda}(d(w,x))^q dm_n^{\otimes 4}(x,y,z,w)\\
		=
		\begin{cases}
		o_\lambda(R^2m_n(B_R)^2), \,\, \text{if }p=q=1,\\
		o_\lambda(m_n(B_R)^2), \quad \text{if }p+q>2.
		\end{cases}
		\end{multline}
	\end{lemma}
	
	\begin{proof}
		As in the proof of \autoref{lem:contractionslambda}, we denote by $I(R)=I_{n,\lambda,p,q}(R)$ the left-hand side of \eqref{eq:contractionbigR},
		we assume without loss of generality $p\leq q$ and drop dependencies of constants on $n,\lambda,p,q$
		to lighten notation, and we apply GM-QM inequality to the last two factors of the integrand:
		\begin{equation*}
		I(R) \lesssim \int_{B_R^4}
		|F_{n,\lambda}(d(x,y))|^p |F_{n,\lambda}(d(y,z))|^q
		|F_{n,\lambda}(d(z,w))|^{2p} dm_n^{\otimes 4}(x,y,z,w)+\dots,
		\end{equation*}
		where dots stand for an analogous term with higher exponents. 
		Integration variables can now be decoupled: moving to polar coordinates and 
		recalling the uniform bound $|F_{n,\lambda}(r)|\leq e^{-\sigma r}$, $r\geq 0$,
		we control
		\begin{multline*}
		I(R) \lesssim m_n(B_R) \pa{\int_0^R e^{-p\sigma r}\sinh(r)^\sigma dr}\\
		\cdot \pa{\int_0^R e^{-q\sigma r}\sinh(r)^\sigma dr}
		\pa{\int_0^R e^{-2p\sigma r}\sinh(r)^\sigma dr}.
		\end{multline*}
		Since $p\geq 1$, the last factor is always $O(R)$ as $R\to\infty$.
		Similarly, the other integrals into parentheses are $O(R)$ or $O(e^{\sigma R})$
		respectively when the exponent $p,q$ equals 1 or not.
		Thus, recalling that $m_n(B_R)\simeq e^{2\sigma R}$ as $R\to \infty$,
		our estimate suffices to conclude the thesis, since it implies:
		\begin{equation*}
		I(R)\lesssim 
		\begin{cases}
		R e^{4\sigma R} &\text{if } p=q=1\\
		R^2 e^{3\sigma R} &\text{if } 1=p<2\leq q\\
		e^{2\sigma R} &\text{if } p,q\geq 2\\
		\end{cases}
		=
		\begin{cases}
		o(R^2 e^{4\sigma R}) &\text{if } p=q=1\\
		o(e^{4\sigma R}) &\text{otherwise }\\
		\end{cases}.	
		\qedhere
		\end{equation*}
	\end{proof}

	\appendix
	
	\section{Repository of Formulae for Special Functions}\label{sec:specialfunctions}
	
	\subsection{Bessel Functions}
	For $\nu\in [0,\infty)$, the Bessel function of first kind $J_\nu(z)$ is defined as,
	\cite[10.2.2]{Olver2010},
	\begin{equation}\label{eq:besselDEF}
	J_\nu(z)=\left(\frac{z}{2}\right)^\nu \sum_{k=0}^{\infty} \frac{\left(-z^2/4\right)^k}{k ! \Gamma(\nu+k+1)},
	\end{equation}
	where the power series is convergent and analytic on the whole $\C$,
	and when $\nu$ is not an integer the principal branch of $z^\nu$ (with branch cut $(-\infty,0]$)
	is customarily chosen for the prefactor.
	The analytic function $z^{-\nu} J_\nu(z)$ provides a representation
	for the Fourier transform of the sphere as reported in \autoref{eq:berrycov}.
	
	The modified Bessel function of first kind $I_\nu$ (often appearing instead of $J_\nu$ in the references below)
	is simply obtained from $J_\nu$ by the relation $I_\nu(z)=i^{\pm \nu}J_\nu(i^{\mp 1}z)$.
	
	\subsection{Hypergeometric Functions}
	For $a,b,c\in\C$, the hypergeometric function ${}_2F_1(a,b,c,z)$ is defined as the analytic continuation of
	\begin{gather}\label{eq:hypergeomDEF}
	{}_2F_1(a,b,c,z)=\sum_{n=0}^\infty \frac{(a)_n(b_n)}{(c)_nn!}z^n,\quad |z|<1,\\
	\nonumber (a)_0=1,\quad (a)_n=a(a+1)\cdots(a+n-1),
	\end{gather}
	on $\C\setminus [1,\infty]$, \cite[15.2.1]{Olver2010}.
	The following linear transformation properties hold: for $z\notin [0,\infty]$, \cite[15.8.2]{Olver2010},
	\begin{multline}\label{eq:hypergeomlin1}
	\frac{\sin(\pi(b-a))}{\pi\Gamma(c)}{}_2F_1(a,b,c,z)\\
	=\frac{(-z)^{-a}}{\Gamma(b)\Gamma(c-a)\Gamma(a-b+1)}{}_2F_1(a,a-c+1,a-b+1,1/z)\\
	-\frac{(-z)^{-b}}{\Gamma(a)\Gamma(c-b)\Gamma(b-a+1)}{}_2F_1(b,b-c+1,b-a+1,1/z),
	\end{multline}
	and for $z\notin [1,\infty]$, \cite[15.8.1]{Olver2010},
	\begin{multline}\label{eq:hypergeomlin2}
	{}_2F_1(a,b,c,z)=(1-z)^{-a}{}_2F_1(a,c-b,c,z/(z-1))\\
	=(1-z)^{-b}{}_2F_1(c-a,b,c,z/(z-1))
	=(1-z)^{c-a-b}{}_2F_1(c-a,c-b,c,z).
	\end{multline}
	
	The hypergeometric equation
	\begin{equation}\label{eq:hypergeomODE}
	z(1-z) \frac{d^2 f}{d z^2}+(c-(a+b+1) z) \frac{d f}{d z}-a b f=0
	\end{equation}
	is a complex-valued ODE with regular singularities at $z=0,1,\infty$. 
	When $c,a-b,c-a-b$ are not integers, the above power series provides two linearly independent solutions
	close to $z=0$,
	\begin{equation*}
	f_1(z)={}_2F_1(a,b,c,z),\quad f_2(z)=z^{1-c}{}_2F_1(a-c+1,b-c+1,2-c,z),
	\end{equation*}
	and the aforementioned analytic extension of ${}_2F_1(a,b,c,z)$ solves \ref{eq:hypergeomODE} on $(-\infty,0]$.
	
	\subsection{Relations with Legendre Functions}
	For a particular choice of parameters, the hypergeometric function provides a representation
	for Legendre functions (we are interested in particular to those of the first kind).
	For $x\in (0,\infty)$, $\mu,\nu\in \C$, \cite[14.3.6]{Olver2010} and \cite[pag. 122, (3)]{Erdelyi1981},
	\begin{equation}\label{eq:legendreDEF}
	P^\mu_\nu(x)=\frac1{\Gamma(1-\mu)}\pa{\frac{x+1}{x-1}}^{\mu/2}{}_2F_1(\nu+1,-\nu,1-\mu,(1-x)/2).
	\end{equation}
	We have the following integral representation for the Legendre function $P^\mu_\nu$ when $\re\mu<1/2$, $r>0$,
	\cite[pag. 156, (7)]{Erdelyi1981},
	\begin{equation}\label{eq:legendreINT}
	P^\mu_\nu(\cosh(r))=\frac{2^\mu}{\sqrt\pi \sinh(r)^\mu \Gamma(1/2-\mu)}
	\int_0^\pi \frac{\pa{\cosh r+\sinh r \cos t}^{\mu+\nu}}{(\sin t)^{-2\mu}}dt.
	\end{equation}
	For fixed $\mu\in\R^+$, as $\nu\to\infty$, \cite[14.15.13]{Olver2010}
	\begin{equation}\label{eq:legendreandbessel}
	P_\nu^{-\mu}(\cosh r)=\frac{1}{(i\nu)^\mu}\left(\frac{r}{\sinh r}\right)^{1 / 2} J_\mu\left(\left(\nu+\frac{1}{2}\right) ir\right)\left(1+O\left(\frac{1}{\nu}\right)\right)
	\end{equation}
	uniformly in $r\in (0,\infty)$.

\bibliographystyle{plain}

\begin{thebibliography}{10}
	
	\bibitem{Abert2018}
	Miklos Abert, Nicolas Bergeron, and Etienne~Le Masson.
	\newblock Eigenfunctions and random waves in the benjamini-schramm limit.
	\newblock October 2018.
	
	\bibitem{Abramowitz1965}
	Milton Abramowitz, editor.
	\newblock {\em Handbook of mathematical functions, with formulas, graphs, and
		mathematical tables}.
	\newblock National Bureau of Standards Applied Mathematics Series, No. 55. U.
	S. Government Printing Office, Washington, D.C., 1965.
	\newblock Superintendent of Documents.
	
	\bibitem{Adler2007}
	Robert~J. Adler and Jonathan~E. Taylor.
	\newblock {\em Random fields and geometry}.
	\newblock Springer Monographs in Mathematics. Springer, New York, 2007.
	
	\bibitem{Anderson2005}
	James~W. Anderson.
	\newblock {\em Hyperbolic geometry}.
	\newblock Springer Undergraduate Mathematics Series. Springer-Verlag London,
	Ltd., London, second edition, 2005.
	
	\bibitem{Anker2012}
	Jean-Philippe Anker, Vittoria Pierfelice, and Maria Vallarino.
	\newblock The wave equation on hyperbolic spaces.
	\newblock {\em J. Differential Equations}, 252(10):5613--5661, 2012.
	
	\bibitem{Azais2009}
	Jean-Marc Aza\"{\i}s and Mario Wschebor.
	\newblock {\em Level sets and extrema of random processes and fields}.
	\newblock John Wiley \& Sons, Inc., Hoboken, NJ, 2009.
	
	\bibitem{Bergeron2016}
	Nicolas Bergeron.
	\newblock {\em The spectrum of hyperbolic surfaces}.
	\newblock Universitext. Springer, Cham; EDP Sciences, Les Ulis, 2016.
	\newblock Appendix C by Valentin Blomer and Farrell Brumley, Translated from
	the 2011 French original by Brumley [2857626].
	
	\bibitem{Berry1977}
	M.~V. Berry.
	\newblock Regular and irregular semiclassical wavefunctions.
	\newblock {\em Journal of Physics. A. Mathematical and General},
	10(12):2083--2091, 1977.
	
	\bibitem{Berry2002}
	M.~V. Berry and H.~Ishio.
	\newblock Nodal densities of {G}aussian random waves satisfying mixed boundary
	conditions.
	\newblock {\em Journal of Physics. A. Mathematical and General},
	35(29):5961--5972, 2002.
	
	\bibitem{Bleistein1986}
	Norman Bleistein and Richard~A. Handelsman.
	\newblock {\em Asymptotic expansions of integrals}.
	\newblock Dover Publications, Inc., New York, second edition, 1986.
	
	\bibitem{Borthwick2016}
	David Borthwick.
	\newblock {\em Spectral theory of infinite-area hyperbolic surfaces}, volume
	318 of {\em Progress in Mathematics}.
	\newblock Birkh\"{a}user/Springer, [Cham], second edition, 2016.
	
	\bibitem{Cammarota2019}
	Valentina Cammarota.
	\newblock Nodal area distribution for arithmetic random waves.
	\newblock {\em Transactions of the American Mathematical Society},
	372(5):3539--3564, 2019.
	
	\bibitem{Cammarota2016}
	Valentina Cammarota, Domenico Marinucci, and Igor Wigman.
	\newblock On the distribution of the critical values of random spherical
	harmonics.
	\newblock {\em J. Geom. Anal.}, 26(4):3252--3324, 2016.
	
	\bibitem{Canzani2020}
	Yaiza Canzani and Boris Hanin.
	\newblock Local universality for zeros and critical points of monochromatic
	random waves.
	\newblock {\em Communications in Mathematical Physics}, 378(3):1677--1712,
	2020.
	
	\bibitem{Cohen2012}
	S.~Cohen and M.~A. Lifshits.
	\newblock Stationary {G}aussian random fields on hyperbolic spaces and on
	{E}uclidean spheres.
	\newblock {\em ESAIM. Probability and Statistics}, 16:165--221, 2012.
	
	\bibitem{Dalmao2022}
	Federico Dalmao.
	\newblock A note on $3d$-monochromatic random waves and cancellation.
	\newblock August 2022.
	
	\bibitem{Dalmao2021}
	Federico Dalmao, Anne Estrade, and Jos\'{e}~R. Le\'{o}n.
	\newblock On 3-dimensional {B}erry's model.
	\newblock {\em ALEA Lat. Am. J. Probab. Math. Stat.}, 18(1):379--399, 2021.
	
	\bibitem{Dalmao2019}
	Federico Dalmao, Ivan Nourdin, Giovanni Peccati, and Maurizia Rossi.
	\newblock Phase singularities in complex arithmetic random waves.
	\newblock {\em Electron. J. Probab.}, 24:Paper No. 71, 45, 2019.
	
	\bibitem{DNPR}
	Gauthier Dierickx, Ivan Nourdin, Maurizia Rossi, and Giovanni Peccati.
	\newblock Small scale clts for the nodal length of monochromatic waves.
	\newblock {\em Communications in Mathematical Physics}, 2022.
	
	\bibitem{Doukhan2003}
	Paul Doukhan, George Oppenheim, and Murad~S. Taqqu, editors.
	\newblock {\em Theory and applications of long-range dependence}.
	\newblock Birkh\"{a}user Boston, Inc., Boston, MA, 2003.
	
	\bibitem{Erdelyi1981}
	Arthur Erd\'{e}lyi, Wilhelm Magnus, Fritz Oberhettinger, and Francesco~G.
	Tricomi.
	\newblock {\em Higher transcendental functions. {V}ols. {I}, {II}}.
	\newblock McGraw-Hill Book Co., Inc., New York-Toronto-London, 1953.
	\newblock Based, in part, on notes left by Harry Bateman.
	
	\bibitem{Gelfand1964}
	I.~M. Gel'fand and G.~E. Shilov.
	\newblock {\em Generalized functions. {V}ol. {I}: {P}roperties and operations}.
	\newblock Academic Press, New York-London, 1964.
	\newblock Translated by Eugene Saletan.
	
	\bibitem{GH80}
	Donald Geman and Joseph Horowitz.
	\newblock Occupation densities.
	\newblock {\em The Annals of Probability}, 12(8):1--67, 1980.
	
	\bibitem{Helgason2000}
	Sigurdur Helgason.
	\newblock {\em Groups and geometric analysis}, volume~83 of {\em Mathematical
		Surveys and Monographs}.
	\newblock American Mathematical Society, Providence, RI, 2000.
	\newblock Integral geometry, invariant differential operators, and spherical
	functions, Corrected reprint of the 1984 original.
	
	\bibitem{Hida1980}
	Takeyuki Hida.
	\newblock {\em Brownian motion}, volume~11 of {\em Applications of
		Mathematics}.
	\newblock Springer-Verlag, New York-Berlin, 1980.
	\newblock Translated from the Japanese by the author and T. P. Speed.
	
	\bibitem{Hislop1994}
	Peter~D. Hislop.
	\newblock The geometry and spectra of hyperbolic manifolds.
	\newblock volume 104, pages 715--776. 1994.
	\newblock Spectral and inverse spectral theory (Bangalore, 1993).
	
	\bibitem{Ingremeau2020}
	Maxime Ingremeau.
	\newblock Lower bounds for the number of nodal domains for sums of two
	distorted plane waves in non-positive curvature.
	\newblock {\em Asian Journal of Mathematics}, 24(3):417--435, 2020.
	
	\bibitem{Jain2017}
	Sudhir~Ranjan Jain and Rhine Samajdar.
	\newblock Nodal portraits of quantum billiards: domains, lines, and statistics.
	\newblock {\em Rev. Modern Phys.}, 89(4):045005, 66, 2017.
	
	\bibitem{Janson1997}
	Svante Janson.
	\newblock {\em Gaussian {H}ilbert spaces}, volume 129 of {\em Cambridge Tracts
		in Mathematics}.
	\newblock Cambridge University Press, Cambridge, 1997.
	
	\bibitem{Jones2001}
	D.~S. Jones.
	\newblock Asymptotics of the hypergeometric function.
	\newblock volume~24, pages 369--389. 2001.
	\newblock Applied mathematical analysis in the last century.
	
	\bibitem{KL2001}
	Marie~F. Kratz and Jos\'{e}~R. Le\'{o}n.
	\newblock Central limit theorems for level functionals of stationary {G}aussian
	processes and fields.
	\newblock {\em J. Theoret. Probab.}, 14(3):639--672, 2001.
	
	\bibitem{KKW2013}
	Manjunath Krishnapur, P\"{a}r Kurlberg, and Igor Wigman.
	\newblock Nodal length fluctuations for arithmetic random waves.
	\newblock {\em Ann. of Math. (2)}, 177(2):699--737, 2013.
	
	\bibitem{Ledoux1991}
	Michel Ledoux and Michel Talagrand.
	\newblock {\em Probability in {B}anach spaces}, volume~23 of {\em Ergebnisse
		der Mathematik und ihrer Grenzgebiete (3) [Results in Mathematics and Related
		Areas (3)]}.
	\newblock Springer-Verlag, Berlin, 1991.
	\newblock Isoperimetry and processes.
	
	\bibitem{Marinucci2011book}
	Domenico Marinucci and Giovanni Peccati.
	\newblock {\em Random fields on the sphere}, volume 389 of {\em London
		Mathematical Society Lecture Note Series}.
	\newblock Cambridge University Press, Cambridge, 2011.
	\newblock Representation, limit theorems and cosmological applications.
	
	\bibitem{Marinucci2016}
	Domenico Marinucci, Giovanni Peccati, Maurizia Rossi, and Igor Wigman.
	\newblock Non-universality of nodal length distribution for arithmetic random
	waves.
	\newblock {\em Geometric and Functional Analysis}, 26(3):926--960, 2016.
	
	\bibitem{Marinucci2015}
	Domenico Marinucci and Maurizia Rossi.
	\newblock Stein-{M}alliavin approximations for nonlinear functionals of random
	eigenfunctions on {$\Bbb{S}^d$}.
	\newblock {\em Journal of Functional Analysis}, 268(8):2379--2420, 2015.
	
	\bibitem{Marinucci2021}
	Domenico Marinucci and Maurizia Rossi.
	\newblock On the correlation between nodal and nonzero level sets for random
	spherical harmonics.
	\newblock {\em Ann. Henri Poincar\'{e}}, 22(1):275--307, 2021.
	
	\bibitem{Marinucci2011a}
	Domenico Marinucci and Igor Wigman.
	\newblock The defect variance of random spherical harmonics.
	\newblock {\em Journal of Physics A Mathematical General}, 44(35):355206,
	September 2011.
	
	\bibitem{Marinucci2011}
	Domenico Marinucci and Igor Wigman.
	\newblock On the area of excursion sets of spherical {G}aussian eigenfunctions.
	\newblock {\em Journal of Mathematical Physics}, 52(9):093301, 21, 2011.
	
	\bibitem{Marinucci2014}
	Domenico Marinucci and Igor Wigman.
	\newblock On nonlinear functionals of random spherical eigenfunctions.
	\newblock {\em Communications in Mathematical Physics}, 327(3):849--872, 2014.
	
	\bibitem{NS2016}
	Fiodor Nazarov and Mikhail Sodin.
	\newblock Asymptotic laws for the spatial distribution and the number of
	connected components of zero sets of gaussian random functions.
	\newblock {\em Journal of Mathematical Physics, Analysis, Geometry},
	12(3):205--278, 2016.
	
	\bibitem{MNPHD}
	Massimo Notarnicola.
	\newblock {\em Probabilistic limit theorems and the geometry of random fields}.
	\newblock PhD thesis, Luxembourg University, 2021.
	
	\bibitem{Notarnicola2022}
	Massimo Notarnicola, Giovanni Peccati, and Anna Vidotto.
	\newblock Functional convergence of berry's nodal lengths: Approximate
	tightness and total disorder.
	\newblock August 2022.
	
	\bibitem{Nourdinfbm}
	Ivan Nourdin.
	\newblock {\em Selected aspects of fractional {B}rownian motion}, volume~4 of
	{\em Bocconi \& Springer Series}.
	\newblock Springer, Milan; Bocconi University Press, Milan, 2012.
	
	\bibitem{NPTRF}
	Ivan Nourdin and Giovanni Peccati.
	\newblock Stein's method on wiener chaos.
	\newblock {\em Probability Theory and Related Fields}, 145(1):75--118, 2009.
	
	\bibitem{Nourdin2012}
	Ivan Nourdin and Giovanni Peccati.
	\newblock {\em Normal approximations with {M}alliavin calculus}, volume 192 of
	{\em Cambridge Tracts in Mathematics}.
	\newblock Cambridge University Press, Cambridge, 2012.
	\newblock From Stein's method to universality.
	
	\bibitem{Nourdin2019}
	Ivan Nourdin, Giovanni Peccati, and Maurizia Rossi.
	\newblock Nodal statistics of planar random waves.
	\newblock {\em Communications in Mathematical Physics}, 369(1):99--151, 2019.
	
	\bibitem{Olver2010}
	Frank W.~J. Olver, Daniel~W. Lozier, Ronald~F. Boisvert, and Charles~W. Clark,
	editors.
	\newblock {\em N{IST} handbook of mathematical functions}.
	\newblock U.S. Department of Commerce, National Institute of Standards and
	Technology, Washington, DC; Cambridge University Press, Cambridge, 2010.
	
	\bibitem{Oravecz2008}
	Ferenc Oravecz, Ze\'{e}v Rudnick, and Igor Wigman.
	\newblock The {L}eray measure of nodal sets for random eigenfunctions on the
	torus.
	\newblock {\em Universit\'{e} de Grenoble. Annales de l'Institut Fourier},
	58(1):299--335, 2008.
	
	\bibitem{Peccati2018}
	Giovanni Peccati and Maurizia Rossi.
	\newblock Quantitative limit theorems for local functionals of arithmetic
	random waves.
	\newblock In {\em Computation and combinatorics in dynamics, stochastics and
		control}, volume~13 of {\em Abel Symp.}, pages 659--689. Springer, Cham,
	2018.
	
	\bibitem{Peccati2020}
	Giovanni Peccati and Anna Vidotto.
	\newblock Gaussian random measures generated by {B}erry's nodal sets.
	\newblock {\em Journal of Statistical Physics}, 178(4):996--1027, 2020.
	
	\bibitem{Rossi2019}
	Maurizia Rossi.
	\newblock The defect of random hyperspherical harmonics.
	\newblock {\em J. Theoret. Probab.}, 32(4):2135--2165, 2019.
	
	\bibitem{Rudnick2008}
	Ze\'{e}v Rudnick and Igor Wigman.
	\newblock On the volume of nodal sets for eigenfunctions of the {L}aplacian on
	the torus.
	\newblock {\em Ann. Henri Poincar\'{e}}, 9(1):109--130, 2008.
	
	\bibitem{Slud1991}
	Eric Slud.
	\newblock Multiple {W}iener-{I}t\^{o} integral expansions for
	level-crossing-count functionals.
	\newblock {\em Probab. Theory Related Fields}, 87(3):349--364, 1991.
	
	\bibitem{Slud1994}
	Eric~V. Slud.
	\newblock M{WI} representation of the number of curve-crossings by a
	differentiable {G}aussian process, with applications.
	\newblock {\em Ann. Probab.}, 22(3):1355--1380, 1994.
	
	\bibitem{Stein1993}
	Elias~M. Stein.
	\newblock {\em Harmonic analysis: real-variable methods, orthogonality, and
		oscillatory integrals}, volume~43 of {\em Princeton Mathematical Series}.
	\newblock Princeton University Press, Princeton, NJ, 1993.
	
	\bibitem{Strichartz1989}
	Robert~S. Strichartz.
	\newblock Harmonic analysis as spectral theory of {L}aplacians.
	\newblock {\em Journal of Functional Analysis}, 87(1):51--148, 1989.
	
	\bibitem{Todino2019}
	Anna~Paola Todino.
	\newblock A quantitative central limit theorem for the excursion area of random
	spherical harmonics over subdomains of {$\Bbb S^2$}.
	\newblock {\em Journal of Mathematical Physics}, 60(2):023505, 33, 2019.
	
	\bibitem{Wigman2022}
	Igor Wigman.
	\newblock On the nodal structures of random fields -- a decade of results.
	\newblock June 2022.
	
	\bibitem{Zelditch2009}
	Steve Zelditch.
	\newblock Real and complex zeros of {R}iemannian random waves.
	\newblock In {\em Spectral analysis in geometry and number theory}, volume 484
	of {\em Contemp. Math.}, pages 321--342. Amer. Math. Soc., Providence, RI,
	2009.
	
\end{thebibliography}

\end{document}